\documentclass{amsart}
\usepackage{latexsym}
\usepackage{amstext}
\usepackage{amsmath}
\usepackage{amssymb}
\usepackage{amsthm}
\usepackage{url}
\usepackage{enumerate}

\newtheorem{theorem}{Theorem}[section]
\newtheorem{lemma}[theorem]{Lemma}

\newtheorem{conjecture}[theorem]{Conjecture}

\theoremstyle{definition}

\newtheorem{example}[theorem]{Example}

\def \co {\mathcal{O}}

\def \kbar {\overline{k}}
\def \Qbar {\overline{\mathbb{Q}}}

\DeclareMathOperator{\ord}{ord}
\DeclareMathOperator{\lcm}{lcm}
\DeclareMathOperator{\Sym}{Sym}
\DeclareMathOperator{\Supp}{Supp}

\DeclareMathOperator{\Div}{Div}
\DeclareMathOperator{\codim}{codim}
\DeclareMathOperator{\dv}{div}

\begin{document}
\bibliographystyle{amsplain}
\title{On the Schmidt Subspace Theorem}
\author{Aaron Levin}
\address{Department of Mathematics\\Michigan State University\\East Lansing, MI 48824}
\curraddr{}
\email{adlevin@math.msu.edu}
\date{}
\begin{abstract}
We study extensions and generalizations of the Schmidt Subspace Theorem in various settings.  In particular, we prove results for algebraic points of bounded degree, giving a sharp version of Schmidt's theorem for quadratic points in the projective plane and a more general result that resolves a conjecture of Schlickewei.
\end{abstract}

\maketitle

\section{Introduction}

The celebrated theorem of Roth in Diophantine approximation describes how closely an algebraic number may be approximated by rational numbers:
\begin{theorem}[Roth \cite{Roth}]
\label{TRoth}
Let $\alpha\in \Qbar$ be an algebraic number.  Let $\epsilon>0$.  Then there are only finitely many rational numbers $\frac{p}{q}\in \mathbb{Q}$ satisfying
\begin{equation*}
\left|\alpha-\frac{p}{q}\right|<\frac{1}{q^{2+\epsilon}}.
\end{equation*}
\end{theorem}
In view of a well-known elementary result of Dirichlet, the exponent $2+\epsilon$ is essentially best possible.  Roth's theorem can be appropriately extended \cite{Lan, Rid} to an arbitrary fixed number field $k$ (in place of $\mathbb{Q}$) and to allow finite sets of absolute values (including non-archimedean ones).

Instead of taking the approximating elements from a fixed number field, a natural variation on Roth's theorem is to consider approximation by algebraic numbers of bounded degree.  In this direction, Wirsing \cite{Wir} proved the following generalization of Roth's theorem.
\begin{theorem}[Wirsing]
Let $\alpha\in \Qbar$.  Let $\delta$ be a positive integer.  Let $\epsilon>0$.  Then there are only finitely many $\beta\in \Qbar$ with $[\mathbb{Q}(\beta):\mathbb{Q}]\leq \delta$ satisfying
\begin{equation*}
\left|\alpha-\beta\right|<\frac{1}{H_{\mathbb{Q}(\beta)}(\beta)^{2\delta+\epsilon}}.
\end{equation*}
\end{theorem}
Here $H_{k}(x)$ is the multiplicative Weil height of $x$ relative to the number field $k$.  The case $\delta=1$ recovers Roth's theorem.  In this formulation of Wirsing's theorem (which we've chosen for simplicity), Schmidt \cite{Sch} showed that the exponent $2\delta+\epsilon$ in the theorem could be improved to $\delta+1+\epsilon$, which is best possible.  However, as in Roth's theorem, Wirsing's theorem can be proved in a more general setting, allowing approximation at a finite set of places and replacing $\mathbb{Q}$ by an arbitrary number field.  In this more general setting the exponent $2\delta+\epsilon$ is sharp (Example~\ref{exinf}).

In a different direction, Schmidt \cite{Sch2} proved a deep higher-dimensional generalization of Roth's theorem, the Subspace Theorem.  We state a general formulation of the theorem, including subsequent improvements by Schlickewei \cite{Schl} to allow for arbitrary number fields and finite sets of places and a result of Vojta \cite{Voj} on the independence of the exceptional hyperplanes from the choices of certain parameters.  For a hypersurface $H\subset \mathbb{P}^n$ over a number field $k$ defined by a homogeneous form $f\in k[x_0,\ldots, x_n]$ of degree $d$ and a place $v$ of $k$, define a local Weil function for $H$ with respect to $v$ by
\begin{equation*}
\lambda_{H,v}(P)=\log \max_i \frac{|x_i|_v^d}{|f(P)|_v},
\end{equation*}
where $P=(x_0,\ldots, x_n)\in\mathbb{P}^n(k)\setminus H$ (see Section \ref{snd} for the definition of $|\cdot|_v$).
\begin{theorem}[The Subspace Theorem]
Let $S$ be a finite set of places of a number field $k$.  For each $v\in S$, let $H_{0,v},\ldots,H_{n,v}\subset \mathbb{P}^n$ be hyperplanes over $k$ in general position.  Let $\epsilon>0$.  Then there exists a finite union of hyperplanes $Z\subset \mathbb{P}^n$, depending only on $\cup_{v\in S}\cup_{1\leq i\leq n} H_{i,v}$, such that  the inequality
\begin{equation*}
\sum_{v\in S}\sum_{i=0}^n \lambda_{H_{i,v},v}(P)< (n+1+\epsilon)h(P)
\end{equation*}
holds for all but finitely many points $P\in \mathbb{P}^n(k)\setminus Z$.
\end{theorem}
Like Roth's fundamental result, the Subspace Theorem has a wide range of important and surprising applications (see \cite{Bilu} for a recent survey).  We will be interested here in proving versions of Schmidt's theorem for points of bounded degree, along the lines of Wirsing's generalization of Roth's theorem.

For quadratic points in the projective plane, we prove the following theorem.
\begin{theorem}
\label{8th}
Let $S$ be a finite set of places of a number field $k$.  Let $L_{0},\ldots,L_{q}\subset\mathbb{P}^2$ be lines over $k$ in general position.  Let $\epsilon>0$.  Then the inequality
\begin{equation*}
\sum_{i=0}^q \sum_{v\in S}\sum_{\substack{w\in M_{k(P)}\\w\mid v}}\lambda_{L_i,w}(P)< (8+\epsilon)h(P)
\end{equation*}
holds for all but finitely many points $P\in \mathbb{P}^2(\kbar)\setminus \cup_{i=1}^qL_i$ satisfying $[k(P):k]\leq 2$.
\end{theorem}
This theorem is sharp in the sense that the constant $8$ on the right-hand side cannot be replaced by any smaller number (Example \ref{exinf}).  If we allow finitely many exceptional lines, then we can give a small improvement.
\begin{theorem}
\label{15th}
Let $S$ be a finite set of places of a number field $k$.  Let $L_{0},\ldots,L_{q}\subset\mathbb{P}^2$ be lines over $k$ in general position.  Let $\epsilon>0$.  Then there exists a finite union of lines $Z\subset \mathbb{P}^2$, depending only on $L_0,\ldots,L_q$, such that the inequality
\begin{equation}
\label{pin}
\sum_{i=0}^q \sum_{v\in S}\sum_{\substack{w\in M_{k(P)}\\w\mid v}}\lambda_{L_i,w}(P)< \left(\frac{15}{2}+\epsilon\right)h(P)
\end{equation}
holds for all but finitely many points $P\in \mathbb{P}^2(\kbar)\setminus Z$ satisfying $[k(P):k]\leq 2$.
\end{theorem}
Conjecturally (see Section \ref{scon}) the inequality \eqref{pin} should hold with $\frac{15}{2}$ replaced by $6$ (and if one allows higher degree exceptional curves, with $\frac{15}{2}$ replaced by $5$).

We also obtain results in higher degrees and higher dimensions, in a more general context.  We recall that generalizations of the Subspace Theorem to projective varieties have been given, independently, by Corvaja and Zannier \cite{CZ} and Evertse and Ferretti \cite{EF}.  We state a version due to Evertse and Ferretti.

\begin{theorem}[Evertse and Ferretti]
\label{EF}
Let $X$ be a projective subvariety of $\mathbb{P}^N$ of dimension $n\geq 1$ defined over a number field $k$.  Let $S$ be a finite set of places of $k$.  For $v\in S$, let $H_{0,v},\ldots,H_{n,v}\subset \mathbb{P}^N$ be hypersurfaces over $k$ such that
\begin{equation*}
X\cap H_{0,v}\cap\cdots\cap H_{n,v}=\emptyset \text{ for } v\in S.
\end{equation*}
Let $\epsilon>0$.  Then there exists a proper Zariski-closed subset $Z\subset X$ such that for all points $P\in X(k)\setminus Z$,
\begin{equation*}
\sum_{v\in S}\sum_{i=0}^n \frac{\lambda_{H_{i,v},v}(P)}{\deg H_{i,v}}< (n+1+\epsilon)h(P).
\end{equation*}
\end{theorem}

Along the same lines, we prove the following theorem for points of bounded degree.
\begin{theorem}
\label{genth}
Let $X$ be a projective subvariety of $\mathbb{P}^N$ of dimension $n\geq 1$ defined over a number field $k$.  Let $S$ be a finite set of places of $k$.  For $v\in S$, let $H_{0,v},\ldots,H_{n,v}\subset \mathbb{P}^N$ be hypersurfaces over $k$ such that
\begin{equation*}
X\cap H_{0,v}\cap\cdots\cap H_{n,v}=\emptyset \text{ for } v\in S.
\end{equation*}
Let $\delta\geq 1$ be an integer and $\epsilon>0$.  Then the inequality
\begin{equation}
\label{Schineq}
\sum_{v\in S}\sum_{\substack{w\in M_{k(P)}\\w\mid v}}\sum_{i=0}^n \frac{\lambda_{H_{i,v},w}(P)}{\deg H_{i,v}}< \left((\delta n)^2\left(\frac{\delta n-1}{2\delta n-3}\right)+\epsilon\right)h(P)
\end{equation}
holds for all but finitely many points $P\in X(\kbar)\setminus \cup_{v\in S}\cup_{i=0}^nH_{i,v}$ satisfying $[k(P):k]\leq \delta$.
\end{theorem}
In fact, we prove a somewhat more general theorem (Theorem \ref{galg}).  In the case where the $H_{i,v}$ are hyperplanes and $X=\mathbb{P}^n=\mathbb{P}^N$, Schlickewei conjectured \cite[Conjecture 5.1]{Schl2} that there exists a constant $c(\delta,n)$, depending only on $\delta$ and $n$, and a finite union of hyperplanes $Z\subset\mathbb{P}^n$ such that the left-hand side of the inequality \eqref{Schineq} is bounded by $(c(\delta,n)+\epsilon)h(P)$ for all points $P\in \mathbb{P}^n(\kbar)\setminus Z$ satisfying $[k(P):k]\leq \delta$.  Thus, Theorem \ref{genth} proves Schlickewei's conjecture as a special case.

Previous work studying Schmidt's theorem for algebraic points appears to be limited to partial results of Locher and Schlickewei \cite[Th.\ 5.3]{Schl2} and Ru and Wang \cite{RW}.  We refer the reader to \cite{Schl2} for the statement of Locher and Schlickewei's technical result.  We now describe the result of Ru and Wang.  For a linear form $l=\sum_{i=0}^n a_ix_i$, let $\bigodot_t l$ denote the $t$-th fold symmetric tensor product of $l$.  This is an element in a vector space of dimension $\binom{n+t}{t}$.
\begin{theorem}[Ru-Wang]
Let $S$ be a finite set of places of a number field $k$.  Let $H_1,\ldots, H_q\subset \mathbb{P}^n$ be hyperplanes defined by linear forms $l_1,\ldots,l_q$ over $k$.  Let $\delta$ be a positive integer.  Suppose that for any positive integer $t\leq \delta$, any $\binom{n+t}{t}$ distinct elements of $\{\bigodot_t l_1,\ldots,\bigodot_t l_q\}$ are linearly independent.  Let $\epsilon>0$.  Then
\begin{equation*}
\sum_{v\in S}\sum_{\substack{w\in M_{k(P)}\\w\mid v}}\sum_{i=0}^n \lambda_{H_i,w}(P)<\left(2\binom{n+\delta}{\delta}-2+\epsilon\right)h(P)+O(1)
\end{equation*}
for all points $P\in \mathbb{P}^n(\kbar)\setminus \cup_{i=1}^qH_i$ satisfying $[k(P):k]\leq \delta$.
\end{theorem}

Even if $H_1,\ldots, H_q$ are in general position, for any $t\geq 2$, $\bigodot_t l_1,\ldots,\bigodot_t l_q$ may have many nontrivial relations \cite[Ex.\ 1]{RW}.
For the purposes of comparison with Theorems \ref{8th} and \ref{genth}, when the hypotheses of the Ru-Wang theorem are satisfied it gives a better bound only in the following cases:  $n=2, \delta\neq 2$; $\delta=2, n\neq 2$; $n=3, \delta\leq 6$; $\delta=3, n\leq 6$.

We now discuss the analogous topics and results in Nevanlinna theory.  As discovered initially by Osgood and Vojta, there is a striking correspondence between statements in Diophantine approximation and statements in Nevanlinna theory.  We refer the reader to \cite{Voj3} for Vojta's dictionary between the two subjects as well as the basic notation and definitions from Nevanlinna theory used below.  Under the Diophantine-Nevanlinna correspondence, Roth's theorem is analogous to Nevanlinna's Second Main Theorem.

\begin{theorem}[Nevanlinna's Second Main Theorem]
Let $f$ be a meromorphic function on $\mathbb{C}$ and let $a_1,\ldots, a_q\in \mathbb{C}$ be distinct numbers.  Then for all $\epsilon>0$,
\begin{equation*}
\sum_{i=1}^qm_f(a_i,r)\leq (2+\epsilon)T_f(r)
\end{equation*}
for all $r>0$ outside a set of finite Lebesgue measure.
\end{theorem}
The proximity function $m_f$ is analogous to ($-\log$ of) the left-hand side of the inequality in Roth's theorem, while the characteristic function $T_f$ is analogous to the height $h$.  Continuing the analogy, Schmidt's Subspace Theorem corresponds to the Second Main Theorem of Cartan.  In fact, to obtain the precise analogue of Schmidt's theorem, one must use the following form of Cartan's theorem due to Vojta \cite{Voj4}.

\begin{theorem}[Cartan's Second Main Theorem]
\label{VC}
Let $H_1,\ldots H_q$ be hyperplanes in $\mathbb{P}^n$ with corresponding Weil functions $\lambda_{H_1},\ldots,\lambda_{H_q}$.  Then there exists a finite union of hyperplanes $Z\subset\mathbb{P}^n$ such that for any $\epsilon >0$ and any nonconstant holomorphic map $f:\mathbb{C}\to \mathbb{P}^n$ with $f(\mathbb{C})\not\subset Z$, the inequality
\begin{equation}
\int_{0}^{2\pi} \max_I \sum_{i \in I} \lambda_{H_i}(f(re^{i\theta}))\frac{d\theta}{2\pi} \leq (n+1+\epsilon)T_f(r)
\end{equation}
holds for all $r$ outside a set of finite Lebesgue measure, where the max is taken over subsets $I \subset \{1,\ldots,q\}$ such that the hyperplanes $H_i,i \in I$, are in general position.
\end{theorem}

We also note that Evertse and Ferretti's Theorem \ref{EF} was proven in the context of Nevanlinna theory by Ru \cite{Ru}.

Wirsing's theorem corresponds to a version of the Second Main Theorem for so-called algebroid functions (see \cite{Ru2}).  Nevanlinna theory for algebroid functions was developed in the 1920's and 1930's by, among others, Selberg \cite{Sel1,Sel2,Sel3}, Ullrich \cite{Ull}, and Valiron \cite{Val}, resulting in the following Second Main Theorem for algebroid functions.

\begin{theorem}
Let $f$ be a $\delta$-valued algebroid function in $|z|<\infty$.  Let $a_1,\ldots, a_q\in \mathbb{C}$ be distinct numbers.  Then for all $\epsilon>0$, 
\begin{equation*}
\sum_{i=1}^qm_f(a_i,r)\leq (2\delta+\epsilon)T_f(r)
\end{equation*}
for all $r>0$ outside a set of finite Lebesgue measure.
\end{theorem}

More generally, in higher dimensions Stoll proved the following theorem \cite[Th.~4.3]{Ru2}.

\begin{theorem}[Stoll]
Let $\pi:M\to \mathbb{C}$ be an (analytic) $\delta$-sheeted covering.  Let $f:M\to\mathbb{P}^n$ be a holomorphic map.  Let $H_1,\ldots, H_q$ be hyperplanes in $\mathbb{P}^n$ in general position.  Then for all $\epsilon>0$,
\begin{equation*}
\sum_{i=1}^qm_f(H_i,r)\leq (2\delta n+\epsilon)T_f(r)
\end{equation*}
for all $r>0$ outside a set of finite Lebesgue measure.
\end{theorem}

The techniques of this paper allow one to prove, for instance, a slightly improved version of Stoll's result in the case $\delta=n=2$.

\begin{theorem}
Let $\pi:M\to \mathbb{C}$ be a $2$-sheeted covering.  Let $f:M\to\mathbb{P}^2$ be a holomorphic map.  Let $L_1,\ldots, L_q$ be lines in $\mathbb{P}^2$ in general position.  Then for all $\epsilon>0$,
\begin{enumerate}
\item 
\begin{equation*}
\sum_{i=1}^qm_f(L_i,r)\leq (8+\epsilon)T_f(r)
\end{equation*}
for all $r>0$ outside a set of finite Lebesgue measure.
\item  There exists a finite union of lines $Z\subset \mathbb{P}^2$ such that either $f(M)\subset Z$ or
\begin{equation*}
\sum_{i=1}^qm_f(L_i,r)\leq \left(\frac{15}{2}+\epsilon\right)T_f(r)
\end{equation*}
for all $r>0$ outside a set of finite Lebesgue measure.
\end{enumerate}
\end{theorem}

This theorem is, of course, the analogue in Nevanlinna theory of Theorems \ref{8th} and \ref{15th}.

The proofs of our arithmetic results depend on the Schmidt Subspace Theorem and formal properties of Weil functions and height functions.  Since the formal properties of Weil functions and height functions carry over without any change to the corresponding functions in Nevanlinna theory, one may use Cartan's Second Main Theorem (Theorem \ref{VC}), in the same manner as Schmidt's theorem is used in the proofs given here, to obtain analogues in Nevanlinna theory of all of our arithmetic results.  This is, by now, well known and straightforward so we will omit the details.

Finally, we give a brief outline of the rest of the paper.  In the next section we give the necessary definitions and background material.  In Section \ref{snum} we prove an extension of Theorem \ref{EF} to numerically equivalent ample divisors.  This is based on the idea that divisors that are numerically equivalent are very close to being linearly equivalent.  In Section \ref{slog} we study the relations between various versions of Schmidt's theorem for numerically equivalent ample divisors where:
\begin{enumerate}
\item  The divisors are in $m$-subgeneral position.
\label{sa}
\item  The dimension of the exceptional set $Z$ is controlled.
\label{sb}
\item  The points are defined over fields of bounded degree.
\label{sc}
\end{enumerate}
We show that results of type \eqref{sa} imply results of type \eqref{sb}, which in turn imply results of type \eqref{sc}.  In Section \ref{sbound} we prove a version of Schmidt's theorem for numerically equivalent ample divisors in $m$-subgeneral position, which by the preceding discussion yields a version of Schmidt's theorem for algebraic points (Theorem \ref{galg}).

In Section \ref{salg} we prove a sharp version of Schmidt's theorem for quadratic points in the projective plane.  We give a short overview of the proof.  We first reduce the theorem to a Diophantine approximation problem for rational points on $\Sym^2\mathbb{P}^2$.  This idea was already present in Ru and Wang's paper \cite{RW}, which in turn was based on a technique of Stoll \cite{Stoll} in Nevanlinna theory.  We handle the Diophantine approximation problem by using variations on the technique pioneered by Corvaja and Zannier in \cite{CZ3} and subsequently developed in, among others, the papers \cite{Aut, Aut2, CLZ, CZ,CZ2,Lev}.  The underlying idea, as applied to our specific case, is to embed $\Sym^2\mathbb{P}^2$, or its subvarieties, into projective space and apply the Schmidt Subspace Theorem in an appropriate fashion.  Since we have to handle arbitrary subvarieties of $\Sym^2\mathbb{P}^2$ in this process, we are led to deal with divisors that are not in general position, which leads to several technical difficulties.  Finally, in the last section we discuss some conjectural versions of Schmidt's theorem for algebraic points.

\section{Background Material}
\label{snd}
Let $k$ be a number field and let $\co_k$ denote the ring of integers of $k$.  We have a canonical set $M_k$ of places (or absolute values) of $k$ consisting of one place for each prime ideal $\mathfrak{p}$ of $\mathcal{O}_k$, one place for each real embedding $\sigma:k \to \mathbb{R}$, and one place for each pair of conjugate embeddings $\sigma,\overline{\sigma}:k \to \mathbb{C}$.  For $v\in M_k$, let $k_v$ denote the completion of $k$ with respect to $v$.  We normalize our absolute values so that $|p|_v=p^{-[k_v:\mathbb{Q}_p]/[k:\mathbb{Q}]}$ if $v$ corresponds to $\mathfrak{p}$ and $\mathfrak{p}|p$, and $|x|_v=|\sigma(x)|^{[k_v:\mathbb{R}]/[k:\mathbb{Q}]}$ if $v$ corresponds to an embedding $\sigma$ (in which case we say that $v$ is archimedean).  If $v$ is a place of $k$ and $w$ is a place of a field extension $L$ of $k$, then we say that $w$ lies above $v$, or $w|v$, if $w$ and $v$ define the same topology on $k$.  If $S$ is a set of places of a number field $k$ and $L$ is a finite extension of $k$, then we denote by $S_L$ the set of places of $L$ lying above places in $S$.

We now briefly recall some facts about Weil functions and height functions that will be used throughout.  We note that although the theory of heights is typically developed in the context of Cartier divisors, the theory can be developed for Weil divisors or even arbitrary closed subschemes (see \cite{Sil}).  On projective space, for a point $P=(x_0,\ldots,x_n)\in \mathbb{P}^n(k)$, the height is given by the formula
\begin{equation*}
h(P)=\sum_{v\in M_k} \log \max\{|x_0|_v,\ldots,|x_n|_v\}.
\end{equation*}
More generally, let $X$ be a projective variety defined over a number field $k$.  To every Weil divisor $D\in \Div(X)$ we can associate a height function $h_D:X(\kbar)\to \mathbb{R}$, well-defined up to a bounded function.  Let $D,E\in \Div(X)$.  Height functions satisfy the following properties:
\begin{enumerate}
\item  $h_{D+E}=h_D+h_E+O(1)$.
\label{ea}
\item  If $\phi:Y\to X$ is a morphism of projective varieties with $\phi(Y)\not\subset \Supp D$, then $h_{\phi^*D}=h_D\circ \phi+O(1)$.
\item  If $D$ is effective, then $h_D(P)\geq O(1)$ for all $P\in X(\kbar)\setminus \Supp D$.
\label{ec}
\end{enumerate}

Similarly, for every divisor $D\in \Div(X)$ and every place $v\in M_k$ we can associate a local Weil function (or local height function) $\lambda_{D,v}:X(k)\setminus \Supp D\to \mathbb{R}$.  When $D$ is effective, the Weil function $\lambda_{D,v}$ gives a measure of the $v$-adic distance of a point to $D$, being large when the point is close to $D$.  As noted in the introduction, if $D$ is a hypersurface in $\mathbb{P}^n$ defined by a homogeneous polynomial $f\in k[x_0,\ldots,x_n]$ of degree $d$, then a Weil function for $D$ is given by
\begin{equation*}
\lambda_{D,v}(P)=\log \max_i \frac{|x_i|_v^d}{|f(P)|_v},
\end{equation*}
where $P=(x_0,\cdots,x_n)\in \mathbb{P}^n(k)\setminus \Supp D$.

Local Weil functions satisfy analogues of properties \eqref{ea}--\eqref{ec} above:  for all $P\in X(k)$ where the relevant Weil functions are defined,
\begin{enumerate}
\item  $\lambda_{D+E,v}(P)=\lambda_{D,v}(P)+\lambda_{E,v}(P)+O(1)$.
\item  If $\phi:Y\to X$ is a morphism of projective varieties with $\phi(Y)\not\subset \Supp D$, then $\lambda_{\phi^*D,v}(P)=\lambda_{D,v}(\phi(P))+O(1)$.
\item  If $D$ is effective, then $\lambda_{D,v}(P)\geq O(1)$.
\end{enumerate}

Let $S$ be a finite set of places in $M_k$.  For $P\in X(\kbar)$ we define the proximity function $m_{D,S}(P)$ by
\begin{equation*}
m_{D,S}(P)=\sum_{v\in S}\sum_{\substack{w\in M_{k(P)}\\w|v}}\lambda_{D,w}(P).
\end{equation*}
This is well-defined up to $O(1)$.  Height functions can be decomposed into a sum of local Weil functions, and in particular we have the inequality
\begin{equation*}
m_{D,S}(P)\leq h_D(P)+O(1)
\end{equation*}
for all $P\in X(\kbar)\setminus \Supp D$.

If $D_1,\ldots, D_q$ are effective (Cartier or Weil) divisors on a projective variety $X$, then we say that $D_1,\ldots, D_q$ are in $m$-subgeneral position if $q>m$ and for any subset $I\subset \{1,\ldots, q\}$, $|I|=m+1$, we have $\cap_{i\in I}\Supp D_i=\emptyset$.  We say that the divisors are in general position if for any subset $I\subset \{1,\ldots, q\}$, $|I|\leq \dim X+1$, we have $\codim \cap_{i\in I}\Supp D_i\geq |I|$, where we set $\dim \emptyset =-1$.

Suppose now that $X$ is regular in codimension one.  If $D,E\in \Div(X)$, then we define
\begin{align*}
\gcd(D,E)&=\sum \min\{\ord_FD,\ord_FE\}F,\\
\lcm(D,E)&=\sum \max\{\ord_FD,\ord_FE\}F,
\end{align*}
where the sums run over all prime divisors $F$ in $\Div(X)$ and $\ord_FD$ denotes the coefficient of $F$ in $D$.  Recall that we can naturally associate a Weil divisor to every Cartier divisor on $X$.  Thus, in this context we can also define the $\gcd$ and $\lcm$ for Cartier divisors (though the result may only be a Weil divisor).  We also define the vector space $L(D)=\{f \in \kbar(X)\mid\dv(f)\geq -D\}$, where $\dv(f)$ is the divisor associated to $f$, and $l(D)=\dim_{\kbar}L(D)$.  We use $\sim$ to denote linear equivalence of divisors and $\equiv$ to denote numerical equivalence of divisors.

\section{Schmidt's theorem for numerically equivalent ample divisors}
\label{snum}
In this section we prove a slight generalization of the theorem of Evertse and Ferretti (Theorem \ref{EF}) in the setting of numerically equivalent ample divisors, or more generally, ample divisors that have a common multiple up to numerical equivalence.  It will be convenient first to give a slight reformulation of Evertse and Ferretti's theorem.

\begin{theorem}[Evertse and Ferretti, reformulated]
\label{Ev}
Let $X$ be a projective variety of dimension $n$ defined over a number field $k$.  Let $S$ be a finite set of places of $k$.  For each $v\in S$, let $D_{0,v},\ldots, D_{n,v}$ be effective Cartier divisors on $X$, defined over $k$, in general position.  Suppose that there exists an ample Cartier divisor $A$ on $X$ and positive integers $d_{i,v}$ such that $D_{i,v}\sim d_{i,v}A$ for all $i$ and for all $v\in S$.  Let $\epsilon>0$.  Then there exists a proper Zariski-closed subset $Z\subset X$ such that for all points $P\in X(k)\setminus Z$,
\begin{equation*}
\sum_{v\in S}\sum_{i=0}^n \frac{\lambda_{D_{i,v},v}(P)}{d_{i,v}}< (n+1+\epsilon)h_A(P).
\end{equation*}
\end{theorem}

\begin{proof}
Let $N$ be a positive integer such that $NA$ is very ample and $d_{i,v}$ divides $N$ for all $i$ and all $v\in S$.  Let $M=l(NA)-1$ and let $\phi:X\to \mathbb{P}^M$ be the embedding of $X$ in $\mathbb{P}^M$ associated to $NA$.  Then since $\frac{N}{d_{i,v}}D_{i,v}\sim NA$, $\frac{N}{d_{i,v}}D_{i,v}=\phi^*H_{i,v}$ for some hyperplane $H_{i,v}$ defined over $k$.  By functoriality and additivity of Weil functions, for $P\in X(k)\setminus \Supp D_{i,v}$ we have (up to $O(1)$)
\begin{equation*}
\lambda_{H_{i,v},v}(\phi(P))=\lambda_{\frac{N}{d_{i,v}}D_{i,v},v}(P)=N\frac{\lambda_{D_{i,v},v}(P)}{d_{i,v}}.
\end{equation*}
Note also that for $P\in X(k)$, $h(\phi(P))=Nh_A(P)+O(1)$.  Substituting these identities into Theorem \ref{EF}, the result then follows immediately.
\end{proof}

We now show that Theorem \ref{Ev} remains true if we replace linear equivalence by numerical equivalence.

\begin{theorem}
\label{mev}
Let $X$ be a projective variety of dimension $n$ defined over a number field $k$.  Let $S$ be a finite set of places of $k$.  For each $v\in S$, let $D_{0,v},\ldots, D_{n,v}$ be effective Cartier divisors on $X$, defined over $k$, in general position.  Suppose that there exists an ample Cartier divisor $A$ on $X$ and positive integers $d_{i,v}$ such that $D_{i,v}\equiv d_{i,v}A$ for all $i$ and for all $v\in S$.  Let $\epsilon>0$.  Then there exists a proper Zariski-closed subset $Z\subset X$ such that for all points $P\in X(k)\setminus Z$,
\begin{equation*}
\sum_{v\in S}\sum_{i=0}^n \frac{\lambda_{D_{i,v},v}(P)}{d_{i,v}}< (n+1+\epsilon)h_A(P).
\end{equation*}
\end{theorem}
We will need the following result in algebraic geometry, due to Matsusaka \cite{Mat} (see also \cite{Kl}).
\begin{theorem}[Matsusaka]
\label{numample}
Let $A$ be an ample Cartier divisor on a projective variety $X$.  Then there exists a positive integer $N_0$ such that for all $N\geq N_0$, and any Cartier divisor $D$ with $D\equiv NA$, $D$ is very ample.
\end{theorem}

\begin{proof}[Proof of Theorem \rm\ref{mev}]
Let $d$ be the least common multiple of the integers $d_{i,v}$.  Replacing $D_{i,v}$ by $\frac{dD_{i,v}}{d_{i,v}}$ for all $i$, $A$ by $dA$, and using the additivity of Weil functions and heights (up to bounded functions), we see that it suffices to prove the case where $d_{i,v}=1$ for all $i$ and all $v\in S$, i.e., $D_{i,v}\equiv A$ for all $i$ and all $v\in S$.  Let $\epsilon>0$.  Let $N_0$ be the integer of Theorem \ref{numample} for $A$.  By standard properties of Weil functions and heights, there exists an integer $N>N_0$ such that
\begin{equation}
\label{Neq}
\frac{N_0}{N}\sum_{v\in S}\sum_{i=0}^n\lambda_{D_{i,v},v}(P)< \epsilon h_A(P)+O(1) 
\end{equation}
for all $P\in X(k)\setminus \cup_{i,v}\Supp D_{i,v}$.

By the choice of $N_0$, we have that $NA-(N-N_0)D_{i,v}$ is very ample for all $i$ and all $v\in S$.  Since for $v\in S$ the divisors $D_{0,v},\ldots,D_{n,v}$ are in general position and $NA-(N-N_0)D_{i,v}$ is very ample for all $i$, there exist effective divisors $E_{i,v}$ such that $NA\sim (N-N_0)D_{i,v}+E_{i,v}$, and for all $v\in S$, the divisors $(N-N_0)D_{0,v}+E_{0,v},\ldots, (N-N_0)D_{n,v}+E_{n,v}$ are in general position.  We may now apply Theorem~\ref{Ev} to the linearly equivalent divisors $(N-N_0)D_{i,v}+E_{i,v}$, $i=0,\ldots,n$, $v\in S$, and $NA$.  We obtain
\begin{align*}
\sum_{v\in S}\sum_{i=0}^n \lambda_{(N-N_0)D_{i,v}+E_{i,v},v}(P)< (n+1+\epsilon)h_{NA}(P)
\end{align*}
for all $P\in X(k)\setminus Z$ for some proper Zariski-closed subset $Z$ of $X$ containing the supports of all $E_{i,v}$, $v\in S$, $i=0,\ldots, n$.  Using additivity and that $\lambda_{E_{i,v},v}$ is bounded from below outside of the support of $E_{i,v}$, we obtain
\begin{equation*}
\sum_{v\in S}\sum_{i=0}^n \left(1-\frac{N_0}{N}\right)\lambda_{D_{i,v},v}(P)< (n+1+\epsilon)h_A(P)+O(1)
\end{equation*}
for all $P \in X(k)\setminus Z$.  Using (\ref{Neq}) we find
\begin{equation}
\label{eq1}
\sum_{v\in S}\sum_{i=0}^n \lambda_{D_{i,v},v}(P)< (n+1+2\epsilon)h_A(P)+O(1)
\end{equation}
for all $P\in X(k)\setminus Z$.
Since there are only finitely many $k$-rational points of bounded height (with respect to $A$), after adding finitely many points to $Z$, we see that (\ref{eq1}) holds with $2\epsilon$ replaced by $3\epsilon$, and without the $O(1)$ term.
\end{proof}

\section{Logical relations between some extensions of Schmidt's theorem}
\label{slog}

In this section we make precise the following implications:
\begin{enumerate}
\item  Schmidt's theorem for numerically equivalent ample divisors in $m$-subgeneral position\\
\begin{equation*}
\Downarrow
\end{equation*}
\label{ia}
\item \eqref{ia} + the dimension of the exceptional set $Z$ is controlled\\
\begin{equation*}
\Downarrow
\end{equation*}
\label{ib}
\item \eqref{ib} for algebraic points of bounded degree
\label{ic}
\end{enumerate}
In the next section we will prove a theorem of type \eqref{ia}, and hence obtain results as in \eqref{ib} and \eqref{ic} above.  The implication \eqref{ia} $\Rightarrow$ \eqref{ib} will follow from repeated application of \eqref{ia} to exceptional sets.  The implication \eqref{ib} $\Rightarrow$ \eqref{ic} will follow from an argument reducing Diophantine approximation problems for algebraic points of bounded degree on $X$ to Diophantine approximation problems for rational points on symmetric powers of $X$.

Let $m\geq n$ be positive integers.  Let $X$ be a projective variety of dimension $n$ defined over a number field $k$.  Let $S$ be a finite set of places of $k$.  For each $v\in S$, let $D_{0,v},\ldots, D_{m,v}$ be effective Cartier divisors on $X$, defined over $k$, in $m$-subgeneral position, i.e., $\cap_{i=0}^m \Supp D_{i,v}=\emptyset$.  Let $D=\sum_{v\in S}\sum_{i=0}^m D_{i,v}$.  Suppose that there exists an ample Cartier divisor $A$ on $X$, defined over $k$, and positive integers $d_{i,v}$ such that $D_{i,v}\equiv d_{i,v}A$ for all $i$ and for all $v\in S$.  For $c\in\mathbb{R}$ and $\delta\in \mathbb{N}$, define the geometric exceptional set $Z=Z\left(c,\delta, m, n,k, S, X, A, \{D_{i,v}\}\right)$ to be the smallest Zariski-closed subset of $X$ such that the set

\begin{equation*}
\left\{P\in X(\kbar)\setminus \Supp D\mid [k(P):k]\leq \delta,  \sum_{v\in S}\sum_{\substack{w\in M_{k(P)}\\w\mid v}} \sum_{i=0}^m\frac{\lambda_{D_{i,v},w}(P)}{d_{i,v}}>c h_A(P)\right\}\setminus Z
\end{equation*}
is finite for all choices of the Weil and height functions.

We now define
\begin{multline*}
C(\delta, l,m,n)=\sup \{c\mid \dim Z\left(c,\delta, m, n,k, S, X, A, \{D_{i,v}\}\right)\geq l \\
\text{ for some $k$, $S$, $X$, $A$, $\{D_{i,v}\}$ as above}\}.
\end{multline*}
Additionally, it will be useful to define and study the particular cases for rational points
\begin{align*}
A(m,n)&=C(1,n,m,n),\\
B(l,m,n)&=C(1,l,m,n).
\end{align*}

It is also useful to consider a variation of the above definitions.  Under the same hypotheses, let $D_1,\ldots, D_q$ be effective Cartier divisors on $X$, defined over $k$, in $m$-subgeneral position and let $D'=\sum_{i=1}^q D_i$.  Suppose that there exists an ample Cartier divisor $A$ on $X$, defined over $k$, and positive integers $d_1,\ldots,d_q$ such that $D_i\equiv d_iA$ for all $i$.  For $c'\in \mathbb{R}$, define the set $Z'=Z'\left(c',\delta, m, n,k, S, X, A, D_1,\ldots, D_q\right)$  to be the smallest Zariski-closed subset of $X$ such that the set

\begin{equation}
\label{exc2}
\left\{P\in X(\kbar)\setminus \Supp D'\mid [k(P):k]\leq \delta,  \sum_{i=1}^q\frac{m_{D_i,S}(P)}{d_{i}}>c' h_A(P)\right\}\setminus Z
\end{equation}
is finite for all choices of the proximity and height functions.

Define
\begin{multline*}
C'(\delta, l,m,n)=\sup \{c'\mid \dim Z'\left(c',\delta, m, n,k, S, X, A, D_1,\ldots, D_q\}\right)\geq l \\
\text{ for some $k$, $S$, $X$, $A$, $D_1,\ldots, D_q$ as above}\}.
\end{multline*}
As before, we define
\begin{align*}
A'(m,n)&=C'(1,n,m,n),\\
B'(l,m,n)&=C'(1,l,m,n).
\end{align*}

We will show in the next section that $C(\delta, l,m,n)$ and $C'(\delta, l,m,n)$ are always finite.

\begin{theorem}
\label{compth}
For all positive integers $\delta,l,m,n$ we have the following inequalities:
\begin{equation*}
B'(l,m,n)\leq B(l,m,n),
\end{equation*}
\begin{align*}
B(l,m,n)&\leq \max_{l\leq j\leq n}A(m,j), & B'(l,m,n)&\leq \max_{l\leq j\leq n}A'(m,j),\\
C(\delta, l,m,n)&\leq B(l,\delta m, \delta n)\leq \max_{l\leq j\leq \delta n}A(\delta m,j), & C'(\delta, l,m,n)&\leq B'(l,\delta m, \delta n)\leq \max_{l\leq j\leq \delta n}A'(\delta m,j).
\end{align*}
\end{theorem}
\begin{proof}

We first prove that $B'(l,m,n)\leq B(l,m,n)$.  For $P\in X(k)\setminus \cup_{i=1}^q\Supp D_i$, we have
\begin{equation}
\label{seq}
\sum_{i=1}^q \frac{m_{D_i,S}(P)}{d_{i}}=\sum_{v\in S}\sum_{i=1}^q \frac{\lambda_{D_i,v}(P)}{d_{i}}.
\end{equation}
Since the divisors $D_i$ are in $m$-subgeneral position, any point $P\in X(k)$ can be $v$-adically close to at most $m$ of the divisors $D_i$.  This implies that there exists a constant $c$ such that for any $P\in X(k)\setminus \cup_{i=1}^q\Supp D_i$ and any $v\in S$, there are indices $i_1,\ldots, i_m\in \{1,\ldots, q\}$ such that
\begin{equation}
\label{eqred}
\sum_{i=1}^q \frac{\lambda_{D_i,v}(P)}{d_{i}}<\sum_{j=1}^m \frac{\lambda_{D_{i_j},v}(P)}{d_{i_j}}+c.
\end{equation}
Using this fact and \eqref{seq}, the result follows.

For the other inequalities, we prove only the notationally simpler inequalities in the right-hand column (the proofs of the remaining inequalities are nearly identical).  Let $c',l,m, n,k, S, X, A, D_1,\ldots, D_q, D$ be as above and let $Z'$ be the exceptional set in \eqref{exc2} with $\delta=1$.  Suppose that $\dim Z'=j\geq l$.  Let $W$ be an irreducible component of $Z'$ (over $\kbar$) with $\dim W=j$.  After extending $k$, we can assume that $W$ is defined over $k$.  The divisors $D_1|_W,\ldots, D_q|_W$ are in $m$-subgeneral position on $W$ and $A|_W$ is ample.  By functoriality of Weil functions and the definitions of $Z'$ and $W$, for any $\epsilon>0$ the set of $P\in W(k)\setminus \Supp D$ such that
\begin{equation*}
\sum_{i=1}^q\frac{m_{D_i|_W,S}(P)}{d_{i}}>(c'-\epsilon) h_{A|_W}(P)
\end{equation*}
is Zariski-dense in $W$.  This means that $A'(m,j)\geq c'$ which implies that $B'(l,m,n)\leq \max_{l\leq j\leq n}A'(m,j)$.

Now let $\delta$ be a positive integer and let $\epsilon>0$.  Consider the set $Z'$ defined to be the Zariski-closure of the set
\begin{equation}
\left\{P\in X(\kbar)\setminus \Supp D\mid [k(P):k]=\delta,  \sum_{i=1}^q\frac{m_{D_i,S}(P)}{d_{i}}>(B'(l,\delta m, \delta n)+\epsilon)h_A(P)\right\}
\end{equation}
(in what follows, we may assume without of loss of generality that $B'(l,\delta m, \delta n)$ is finite).  To prove that $C'(\delta, l,m,n)\leq B'(l,\delta m, \delta n)$ it suffices to show that $\dim Z'<l$ (for all possible choices of the parameters).


Using standard properties of Weil functions and heights, replacing $D_i$ by $\left(\prod_{j\neq i}d_i\right)D_i$ and $A$ by $\left(\prod_{j=1}^qd_i\right)A$, it suffices to prove the case where $d_i=1$ for all $i$.  Let $X^\delta$ denote the $\delta$th power of $X$ and let $X^{(\delta)}$ denote the $\delta$th symmetric power of $X$, i.e., $X^{(\delta)}=X^\delta/S_{\delta}$, the quotient of $X^\delta$ by the natural action of the symmetric group $S_\delta$ acting by permuting the coordinates of $X^\delta$.  Let $\phi:X^\delta\to X^{(\delta)}$ denote the natural map.  For a point $P\in X(\kbar)$ satisfying $[k(P):k]=\delta$, let $P=P_1,\ldots,P_\delta$ denote the $\delta$ distinct conjugates of $P$ over $k$ and let $\psi(P)=(P_1,\ldots,P_{\delta})\in X^\delta$.  Let $\pi_i$, $i=1,\ldots,\delta$, denote the $i$th projection map from $X^\delta$ to $X$.  Let $E_i=\phi_*\pi_1^*(D_i)$, $i=1,\ldots, q$, $E=\sum_{i=1}^qE_i$, and $B=\phi_*\pi_1^*(A)$.  We note that $\phi^*\phi_*\pi_1^*D_i=\sum_{j=1}^\delta \pi_j^*D_i$.  Since the $D_i$ are in $m$-subgeneral position, it is easily seen that the $E_i$ are in $m\delta$-subgeneral position.  Moreover, $D_i\equiv D_j\equiv A$ for all $i$ and $j$ implies that $E_i\equiv E_j\equiv B$ for all $i$ and $j$.  It follows from the fact that $A$ is ample that $B$ is ample, and consequently $E_i$ is ample (and effective) for all $i$.  We note also that if $P\in X(\kbar)$, $[k(P):k]=\delta$, then the point $\phi(\psi(P))\in X^{(\delta)}$ is $k$-rational.  Then from the definition of $B'(l,\delta m, \delta n)$, we have the inequality
\begin{equation}
\label{beq}
\sum_{i=1}^q m_{E_i,S}(P)<(B'(l,\delta m, \delta n)+\epsilon)h_{B}(P), \quad \forall P\in X^{(\delta)}(k)\setminus (Z\cup \Supp E),
\end{equation}
for some Zariski-closed subset $Z$ of $X^{(\delta)}$ satisfying $\dim Z\leq l-1$.  For $P\in X(\kbar)\setminus \Supp D$, $[k(P):k]=\delta$, we have, up to $O(1)$,
\begin{align*}
\sum_{i=1}^q m_{E_i,S}(\phi(\psi(P)))&=\sum_{i=1}^q m_{\phi_*\pi_1^*D_i,S}(\phi(\psi(P)))=\sum_{i=1}^q m_{\phi^*\phi_*\pi_1^*D_i,S}(\psi(P))\\
&=\sum_{i=1}^q m_{\sum_{j=1}^\delta \pi_j^*D_i,S}(\psi(P))=\sum_{i=1}^q \sum_{j=1}^\delta m_{\pi_j^*D_i,S}(\psi(P))\\
&=\sum_{i=1}^q \sum_{j=1}^\delta m_{D_i,S}(\pi_j(\psi(P)))=\sum_{i=1}^q \sum_{j=1}^\delta m_{D_i,S}(P_j)\\
&=\delta\sum_{i=1}^q m_{D_i,S}(P).
\end{align*}

A similar calculation gives $h_B(\phi(\psi(P)))=\delta h_A(P)+O(1)$. It follows that for an appropriate choice of the functions,
\begin{equation*}
\sum_{i=1}^q m_{E_i,S}(\phi(\psi(P)))>(B'(l,\delta m, \delta n)+\epsilon)h_{B}(\phi(\psi(P)))
\end{equation*}
for all $P\in R$, where $R\subset Z'(\kbar)$ is a Zariski-dense set of points in $Z'$ satisfying $[k(P):k]=\delta$ for all $P\in R$.  From \eqref{beq}, we have $\dim \phi(\psi(R))\leq l-1$.  Since $\phi$ is a finite map and $\pi_1(\psi(R))=R$, $\dim \phi(\psi(R))\geq \dim R=\dim Z'$.  We conclude that $\dim Z'\leq l-1$, as desired.


\end{proof}

\section{A bound for the function $A(m,n)$ and a Schmidt theorem for algebraic points}
\label{sbound}

Using the notation from last section, we will prove the bound
\begin{equation*}
A(m,n)\leq
\begin{cases}
n+1 & \text{if }m=n,\\
\frac{m(m-1)(n+1)}{m+n-2} & \text{if } m\geq 2.
\end{cases}
\end{equation*}

If $m=1$ then necessarily $m=n=1$ (we always have $m\geq n$ as the supports of any $n=\dim X$ ample effective divisors on a projective variety $X$ have nontrivial intersection).  The bound $A(n,n)\leq n+1$ is simply a restatement of Theorem \ref{Ev}.  For $m\geq 2$ we prove the following theorem.

\begin{theorem}
\label{degen}
Let $X$ be a projective variety of dimension $n$ defined over a number field $k$.  Let $S$ be a finite set of places of $k$.  For each $v\in S$, let $D_{0,v},\ldots, D_{m,v}$ be effective ample Cartier divisors on $X$, defined over $k$, in $m$-subgeneral position, $m\geq 2$.  Suppose that there exists a Cartier divisor $A$ on $X$, defined over $k$, and positive integers $d_{i,v}$ such that $D_{i,v}\equiv d_{i,v}A$ for all $i$ and for all $v\in S$.  Let $\epsilon>0$.  Then the  inequality
\begin{equation*}
\sum_{v\in S}\sum_{i=0}^m \frac{\lambda_{D_{i,v},v}(P)}{d_{i,v}}< \left(\frac{m(m-1)(n+1)}{m+n-2}+\epsilon\right)h_A(P)+O(1)
\end{equation*}
holds for all $P\in X(k)\setminus \cup_{i,v} \Supp D_{i,v}$.
\end{theorem}

Combined with Theorem \ref{compth}, we find that if $\delta m\geq 2$,
\begin{equation*}
C(\delta, 1,m,n)\leq \max_{1\leq j\leq \delta n}A(\delta m,j)\leq \max_{1\leq j\leq \delta n}\frac{\delta m(\delta m-1)(j+1)}{\delta m+j-2}\leq \frac{\delta m(\delta m-1)(\delta n+1)}{\delta m+\delta n-2}.
\end{equation*}

Explicitly, this gives a version of Schmidt's theorem for algebraic points.
\begin{theorem}
\label{galg}
Let $X$ be a projective variety of dimension $n$ defined over a number field $k$.  Let $S$ be a finite set of places of $k$.  For each $v\in S$, let $D_{0,v},\ldots, D_{m,v}$ be effective ample Cartier divisors on $X$, defined over $k$, in $m$-subgeneral position.  Suppose that there exists a Cartier divisor $A$ on $X$, defined over $k$, and positive integers $d_{i,v}$ such that $D_{i,v}\equiv d_{i,v}A$ for all $i$ and for all $v\in S$.  Let $\delta$ be a positive integer with $\delta m\geq 2$ and $\epsilon>0$.  Then the  inequality
\begin{equation*}
\sum_{v\in S}\sum_{\substack{w\in M_{k(P)}\\w\mid v}}\sum_{i=0}^m \frac{\lambda_{D_{i,v},w}(P)}{d_{i,v}}< \left(\frac{\delta m(\delta m-1)(\delta n+1)}{\delta m+\delta n-2}+\epsilon\right)h_A(P)+O(1)
\end{equation*}
holds for all points $P\in X(\kbar)\setminus \cup_{i,v} \Supp D_{i,v}$ satisfying $[k(P):k]\leq \delta$.
\end{theorem}

If $m=n\geq 2$, then using $A(\delta n,\delta n)\leq \delta n+1$ gives the slightly better bound
\begin{equation*}
C(\delta, 1,n,n)\leq \max_{1\leq j\leq \delta n}A(\delta n,j)\leq (\delta n)^2\frac{\delta n-1}{2\delta n-3}.
\end{equation*}
As a special case, this gives Theorem \ref{genth} from the introduction.  We note that this bound is essentially quadratic in $\delta n$, while by Example \ref{exinf}, $C(\delta,1,n,n)\geq 2\delta n$ (and it's plausible that equality holds).   

Before proving Theorem \ref{degen} we prove a few lemmas.

\begin{lemma}
\label{RR}
Let $X$ be a projective variety of dimension $n$ that is regular in codimension one.  Let $A$ and $D$ be effective big Cartier divisors on $X$ with $D\equiv A$.  Suppose that there exists an integer $N_0$ such that $|NA|$ is basepoint free for $N\geq N_0$.  For $N\geq N_0$ a positive integer, let $\phi_{NA}:X\to \mathbb{P}^{l(NA)-1}$ be the natural morphism associated to $NA$.  Then for $N\geq N_0$ there exists a set $\mathcal{H}$ of $l(NA)$ hyperplanes of $\mathbb{P}^{l(NA)-1}$ in general position such that
\begin{equation*}
\sum_{H\in \mathcal{H}} \phi_{NA}^*H> \left(\frac{A^n}{(n+1)!}N^{n+1}+O(N^n)\right)D.
\end{equation*}
\end{lemma}

\begin{proof}
We first assume that $D\sim A$.  Let $N\geq N_0$ be a positive integer.  Let $V_i=L(NA-iD)$.  We consider the filtration
\begin{equation*}
L(NA)=V_0\supset V_1\supset V_2\supset \cdots \supset V_N\supset V_{N+1}=\{0\}.
\end{equation*}
Let $B$ be a basis of $L(NA)$ obtained by taking a basis of $V_N$ and successively completing this basis to a basis of $V_{N-1}$, $V_{N-2}, \ldots, V_0$.  For every rational function $f\in L(NA)$ there is a corresponding hyperplane $H\subset \mathbb{P}^{l(NA)-1}$ such that $\phi_{NA}^*H=\dv(f)+NA$.  So if $f\in V_i=L(NA-iD)$ and $H$ is the corresponding hyperplane, then $\phi_{NA}^*H\geq iD$.  Let $\mathcal{H}$ be the set of hyperplanes corresponding to the basis $B$.  Since $B$ is a basis of $L(NA)$, the hyperplanes in $\mathcal{H}$ are in general position.  As $D\sim A$, $l(NA-iD)=l((N-i)A)$.  Note also that $A$ is a nef divisor since $|NA|$ is basepoint free.  By Riemann-Roch, for $0\leq i\leq N$,
\begin{equation*}
\dim V_i=l(NA-iD)=l((N-i)A)=(N-i)^n\frac{A^n}{n!}+O(N^{n-1}).
\end{equation*}
It follows that
\begin{align*}
\sum_{H\in \mathcal{H}} \phi_{NA}^*H&\geq \left(\sum_{i=0}^N i\dim V_i/V_{i+1}\right)D=\left(\sum_{i=1}^N \dim V_i\right)D\\
&\geq\left(\sum_{i=1}^N (N-i)^n\frac{A^n}{n!}+O(N^{n-1})\right)D\\
&\geq \left(\frac{A^n}{(n+1)!}N^{n+1}+O(N^n)\right)D,
\end{align*}
proving the theorem in this case.  The case $D\equiv A$ follows from this case using the same method as in Section \ref{snum}.
\end{proof}

We'll also make use of the following linear algebra lemma \cite[Lemma 3.2]{CZ2}.
\begin{lemma}
\label{fil}
Let $V$ be a vector space of finite dimension $d$.  Let $V=W_1\supset W_2\supset \cdots \supset W_h$ and $V=W_1^*\supset W_2^*\supset \cdots \supset W_{h^*}^*$ be two filtrations of $V$.  There exists a basis $v_1,\dots, v_d$ of $V$ which contains a basis of each $W_j$ and $W_j^*$.
\end{lemma}

\begin{lemma}
\label{lhyp}
Let $X$ be a projective variety of dimension $n$ that is regular in codimension one.  Let $A$, $D_1$, and $D_2$ be effective big Cartier divisors on $X$ with $A\equiv D_1\equiv D_2$.  Suppose that there exists an integer $N_0$ such that $|NA|$ is basepoint free for $N\geq N_0$.  For $N\geq N_0$ a positive integer, let $\phi_{NA}:X\to \mathbb{P}^{l(NA)-1}$ be the natural morphism associated to $NA$.  Then for $N\geq N_0$ there exists a set $\mathcal{H}$ of $l(NA)$ hyperplanes of $\mathbb{P}^{l(NA)-1}$ in general position such that
\begin{equation*}
\sum_{H\in \mathcal{H}} \phi^*H> \left(\frac{A^n}{(n+1)!}N^{n+1}+O(N^n)\right)\lcm(D_1,D_2).
\end{equation*}
\end{lemma}

\begin{proof}
Consider the two filtrations of $L(NA)$ coming from looking at the order of vanishing along $D_1$ and $D_2$, as in Lemma \ref{RR}.  Let $B$ be the basis of $L(NA)$ that Lemma \ref{fil} gives with respect to these two filtrations.  Let $\mathcal{H}$ be the corresponding set of hyperplanes in $\mathbb{P}^{l(NA)-1}$.  Then by Lemma \ref{RR} and the definition of $B$, 
\begin{align*}
\sum_{H\in \mathcal{H}} \phi^*H> \left(\frac{A^n}{(n+1)!}N^{n+1}+O(N^n)\right)D_1,\\
\sum_{H\in \mathcal{H}} \phi^*H> \left(\frac{A^n}{(n+1)!}N^{n+1}+O(N^n)\right)D_2.  
\end{align*}
It follows that $\sum_{H\in \mathcal{H}} \phi^*H> \left(\frac{A^n}{(n+1)!}N^{n+1}+O(N^n)\right)\lcm(D_1,D_2)$.
\end{proof}

We now prove Theorem \ref{degen}.

\begin{proof}[Proof of Theorem \ref{degen}]
For $n=1$ the theorem follows immediately from Faltings' theorem on rational points on curves and the classical Diophantine approximation results for curves of genus zero and one.  We assume from now on that $n\geq 2$.  Since for any $v\in S$ the divisors $D_{0,v},\ldots, D_{m,v}$ are in $m$-subgeneral position, by the same reasoning that led to \eqref{eqred}, after reindexing the divisors $D_{i,v}$ it suffices to prove that the slightly weaker inequality
\begin{equation*}
\sum_{v\in S}\sum_{i=1}^m \frac{\lambda_{D_{i,v},v}(P)}{d_{i,v}}< \left(\frac{m(m-1)(n+1)}{m+n-2}+\epsilon\right)h_A(P)+O(1)
\end{equation*}
holds for all $P\in X(k)\setminus \cup_{i,v} \Supp D_{i,v}$.

We first prove the case where $d_{i,v}=1$ for all $i$ and $v$.  Suppose first that $X$ is normal and so, in particular, regular in codimension one.  Let $N$ be a positive integer such that $NA$ is very ample.  Let $\phi=\phi_{NA}:X\to \mathbb{P}^{l(NA)-1}$ be the corresponding embedding.  By Lemma \ref{lhyp}, for each choice of $i$, $j$, and $v$, we have a set $\mathcal{H}_{i,j,v}$ of $l(NA)$ hyperplanes of $\mathbb{P}^{l(NA)-1}$ in general position satisfying
\begin{equation}
\label{eqij}
\sum_{H\in \mathcal{H}_{i,j,v}} \phi^*H> \left(\frac{A^n}{(n+1)!}N^{n+1}+O(N^n)\right)\lcm(D_{i,v},D_{j,v}).
\end{equation}
Moreover, since all of our objects are defined over $k$, the hyperplanes in $\mathcal{H}_{i,j,v}$ may be chosen to be defined over $k$.  Fixing $i$ and $j$ and applying Schmidt's theorem to $\mathbb{P}^{l(NA)-1}$ and the hyperplanes $\mathcal{H}_{i,j,v}$, $v\in S$, we find that
\begin{equation}
\label{eqS}
\sum_{v\in S}\sum_{H\in \mathcal{H}_{i,j,v}} \lambda_{H,v}(P)< \left(l(NA)+\epsilon\right)h(P)+O(1)
\end{equation}
for all $P$ in $\mathbb{P}^{l(NA)-1}(k)$ outside of some finite union of hyperplanes.  By \eqref{eqij} and the functoriality and additivity of Weil functions, for all $P\in X(k)$ outside of a proper closed subset of $X$, we have, up to $O(1)$,
\begin{align}
\label{eqW}
\sum_{v\in S}\sum_{H\in \mathcal{H}_{i,j,v}} \lambda_{H,v}(\phi(P))&=\sum_{v\in S}\sum_{H\in \mathcal{H}_{i,j,v}} \lambda_{\phi^*H,v}(P)\notag\\
&\geq \left(\frac{A^n}{(n+1)!}N^{n+1}+O(N^n)\right)\sum_{v\in S} \lambda_{\lcm(D_{i,v},D_{j,v}),v}(P).
\end{align}
Similarly, up to $O(1)$,
\begin{equation*}
\left(l(NA)+\epsilon\right)h(\phi(P))=\left(l(NA)+\epsilon\right)h_{NA}(P)=N\left(l(NA)+\epsilon\right)h_{A}(P).
\end{equation*}
Since $l(NA)=\frac{N^n}{n!}A^n+O(N^{n-1})$, we have
\begin{equation}
\label{eqh}
\left(l(NA)+\epsilon\right)h(\phi(P))=\left(\frac{A^n}{n!}N^{n+1}+O(N^{n})\right)h_A(P)+O(1).
\end{equation}
Using \eqref{eqW} and \eqref{eqh} and applying \eqref{eqS} to $\phi(P)$, we find that
\begin{equation*}
\left(\frac{A^n}{(n+1)!}N^{n+1}+O(N^n)\right)\sum_{v\in S} \lambda_{\lcm(D_{i,v},D_{j,v}),v}(P)<\left(\frac{A^n}{n!}N^{n+1}+O(N^{n})\right)h_A(P)
\end{equation*}
or
\begin{equation*}
\sum_{v\in S} \lambda_{\lcm(D_{i,v},D_{j,v}),v}(P)<(n+1+O(1/N))h_A(P)
\end{equation*}
for all $P$ in $X(k)$ outside of a proper closed subset of $X$.  Choosing $N$ sufficiently large, we find that
\begin{equation}
\label{fundineq}
\sum_{v\in S} \lambda_{\lcm(D_{i,v},D_{j,v}),v}(P)<\left(n+1+\epsilon\right)h_A(P)
\end{equation}
for all $P$ in $X(k)$ outside of a proper closed subset of $X$.  Summing over all $m(m-1)$ distinct $i,j\in \{1,\ldots, m\}$, we obtain
\begin{equation}
\label{mmeq}
\sum_{v\in S} \sum_{\substack{i,j=1\\i\neq j}}^m\lambda_{\lcm(D_{i,v},D_{j,v}),v}(P)<m(m-1)(n+1+\epsilon)h_A(P)
\end{equation}
for all $P$ in $X(k)$ outside of a proper closed subset of $X$.  We now claim that for any $v\in S$,
\begin{equation}
\label{meq}
\sum_{\substack{i,j=1\\i\neq j}}^m\lcm(D_{i,v},D_{j,v})\geq (m+n-2)\sum_{l=1}^mD_{l,v}.
\end{equation}
Fix $i\in \{1,\ldots, m\}$ and $v\in S$.  First, note that any irreducible component $E$ of $D_{i,v}$ can belong to at most $m-n$ divisors $D_{j,v}$, $j\neq i$.  If not, then $E$ would be contained in the intersection of $\geq m-n+2$ divisors $D_{l,v}$, and the intersection of $E$ with the remaining $\leq m+1-(m-n+2)=n-1$ ample effective divisors would be nonempty.  This would violate the fact that the divisors $D_{l,v}$ are in $m$-subgeneral position.  That any irreducible component $E$ of $D_{i,v}$ can belong to at most $m-n$ divisors $D_{j,v}$, $j\neq i$, implies that
\begin{align*}
\sum_{\substack{j=1\\j\neq i}}^m \lcm(D_{i,v},D_{j,v})&\geq (m-1-(m-n))D_{i,v}+\sum_{\substack{j=1\\j\neq i}}^m D_{j,v}\\
&\geq (n-1)D_{i,v}+\sum_{\substack{j=1\\j\neq i}}^m D_{j,v}.
\end{align*}
Summing over all $i$, we arrive at \eqref{meq}.  Using the additivity of Weil functions, combining \eqref{mmeq} and \eqref{meq}, we obtain that for any $\epsilon>0$, 
\begin{equation}
\label{mmeq2}
\sum_{v\in S} \sum_{i=1}^m\lambda_{D_{i,v},v}(P)<\left(\frac{m(m-1)(n+1)}{m+n-2}+\epsilon\right)h_A(P)
\end{equation}
for all $P$ in $X(k)$ outside of a proper closed subset $Z$ of $X$.

If $X$ isn't normal then we consider the normalization $\pi:\tilde{X}\to X$ and the divisors $\pi^*A$ and $\pi^*D_{i,v}$ for all $i$ and $v$.  Note that for $v\in S$, the divisors $\pi^*D_{0,v},\ldots, \pi^*D_{m,v}$ are again in $m$-subgeneral position.  Since $A$ is ample, $\pi^*A$ is big and $|N\pi^*A|$ is basepoint free for sufficiently large $N$.  Then the argument above gives
\begin{equation*}
\sum_{v\in S} \sum_{i=1}^m\lambda_{\pi^*D_{i,v},v}(P)<\left(\frac{m(m-1)(n+1)}{m+n-2}+\epsilon\right)h_{\pi^*A}(P)
\end{equation*}
for all $P$ in $\tilde{X}(k)$ outside of a proper closed subset $\tilde{Z}$ of $\tilde{X}$.  By functoriality we again obtain \eqref{mmeq2} for all $P$ in $X(k)$ outside of a proper closed subset $Z=\pi(\tilde{Z})$ of $X$.

Note that we can again apply an inequality as in \eqref{mmeq2} to any irreducible component $Z'$ of $Z$ (outside $\cup_{i,v}D_{i,v}$) and the divisors $D_{i,v}$ and $A$ restricted to $Z'$.  Since $\frac{m(m-1)(n+1)}{m+n-2}$ is an increasing function of $n$, by induction we find that \eqref{mmeq2} holds for all but finitely many $P$ in $X(k)\setminus \cup_{i,v}\Supp D_{i,v}$.  This proves the theorem in the case $d_{i,v}=1$ for all $i$ and $v$.  The general case follows easily from this case, however, by applying this case of the theorem to the divisors $\frac{\lcm_{j,v}\{d_{j,v}\}}{d_{i,v}}D_{i,v}$ and $\lcm_{j,v}\{d_{j,v}\}A$.
\end{proof}

\section{Schmidt's Theorem for quadratic points in the plane}
\label{salg}

We begin with a Diophantine approximation result for curves in the projective plane.

\begin{theorem}
\label{thl}
Let $S$ be a finite set of places of a number field $k$.  Let $L_1,\ldots, L_q$ be distinct lines over $k$ in $\mathbb{P}^2$ in $m$-subgeneral position.  Let $D=\sum_{i=1}^qL_i$.  Let $\epsilon>0$.  Let $C$ be a projective curve in $\mathbb{P}^2$.
\begin{enumerate}
\item  The inequality
\begin{equation}
\label{thl1}
m_{D,S}(P)<(2\delta m+\epsilon)h(P)+O(1)
\end{equation}
holds for all points $P\in C(\kbar)\setminus \Supp D\subset \mathbb{P}^2$ satisfying $[k(P):k]\leq \delta$.
\item  If $C$ is not a line, the inequality
\begin{equation}
\label{thl2}
m_{D,S}(P)<((m+1)\delta+\epsilon)h(P)+O(1)
\end{equation}
holds for all points $P\in C(\kbar)\setminus \Supp D\subset \mathbb{P}^2$ satisfying $[k(P):k]\leq \delta$.
\end{enumerate}
\end{theorem}

In the proof, we will use the following generalization of Wirsing's theorem proven by Song and Tucker \cite{ST} based on an inequality of Vojta \cite{Voj2}.

\begin{theorem}[Song-Tucker]
\label{corV}
Let $C$ be a nonsingular curve defined over a number field $k$.  Let $A$ be an ample divisor on $C$ and $D$ a reduced effective divisor on $C$, both defined over $k$.  Let $S$ be a finite set of places of $k$.  Let $\delta$ be a positive integer and let $\epsilon>0$.  Then
\begin{equation}
\label{Vojta2}
m_{D,S}(P)\leq (2\delta+\epsilon) h_A(P)+O(1)
\end{equation}
for all points $P\in C(\kbar)\setminus \Supp D$ with $[k(P):k]\leq \delta$.
\end{theorem}

\begin{proof}[Proof of Theorem \ref{thl}]
Let $C$ be a curve in $\mathbb{P}^2$.  Let $d=\deg C$.  Let $\tilde{C}$ be a normalization of $C$ and let $\phi:\tilde{C}\to \mathbb{P}^2$ denote the composition of the normalization map $\tilde{C}\to C$ with the inclusion map of $C$ into $\mathbb{P}^2$.

We first prove part (a).  The intersection multiplicity of $C$ and $D$ at any point is at most $md$.  So $\phi^*D\leq mdE$ where $E$ is an effective reduced divisor on $\tilde{C}$.  Then for any point $A$ on $\tilde{C}$, we have 
\begin{align*}
m_{D,S}(\phi(P))&\leq mdm_{E,S}(P)+O(1),\\
dh_A(P)&\leq (1+\epsilon)h(\phi(P))+O(1)
\end{align*}
for all $P\in \tilde{C}(\kbar)\setminus \Supp E$.  Since by Theorem \ref{corV}, $m_{E,S}(P)\leq (2\delta+\epsilon)h_A(P)+O(1)$ for all $P\in \tilde{C}(\kbar)\setminus \Supp E$ satisfying $[k(P):k]\leq \delta$, equation \eqref{thl1} follows.

We now prove part (b).  Let $Q\in C(\kbar)\subset \mathbb{P}^2$.  Let $\pi:X\to \mathbb{P}^2$ be the blow-up of $\mathbb{P}^2$ at $Q$ with exceptional curve $E$, and let $C'$ and $D'$ denote the strict transforms of $C$ and $D$, respectively.  Let $\psi:\tilde{C}\to X$ denote the natural induced map.  Let $Q'\in C'(\kbar)$ be a point lying above $Q$.  Then
\begin{equation}
\label{bleq}
\sum_{\substack{\tilde{Q}\in \tilde{C}\\\psi(\tilde{Q})=Q'}}\ord_{\tilde{Q}}(\phi^*D)=(C'.\pi^*D)_{Q'}.
\end{equation}
We first show that
\begin{equation}
\label{eqint}
(C'.\pi^*D)_{Q'}\leq d+(m-1)(C'.E)_{Q'}.
\end{equation}

Let $\mu_P(C)$ denote the multiplicity of $C$ at $P$.  Suppose that $Q'\not\in \Supp D'$.  Since $\pi^*D=D'+lE$ for some $l\leq m$, where $l$ is the number of lines $L_i$ meeting at $Q$, we have $(C'.\pi^*D)_{Q'}=l(C'.E)_{Q'}$.  It is a standard fact that
\begin{equation}
\label{iexc}
C'.E=\sum_{Q'\in C'}(C'.E)_{Q'}=\mu_Q(C)\leq d.
\end{equation}
So  
\begin{equation*}
(C'.\pi^*D)_{Q'}\leq (C'.E)_{Q'}+(l-1)(C'.E)_{Q'}\leq d+(m-1)(C'.E)_{Q'}
\end{equation*}
and \eqref{eqint} holds.  If $Q'$ is in the support of $D'$, then since distinct lines intersect transversally, there is a unique line $L_i$, $i\in \{1,\ldots, q\}$, such that the strict transform $L_i'$ contains $Q'$.   We have $(C'.\pi^*L_i)_{Q'}\leq d$.  Since $Q'\not\in L_j'$, $j\neq i$, we compute 
\begin{equation*}
(C'.\pi^*D)_{Q'}=(C'.\pi^* L_i)_{Q'}+(l-1)(C'.E)_{Q'}\leq d+(m-1)(C'.E)_{Q'}
\end{equation*}
and \eqref{eqint} again holds.

Let $\mu=\max\{\mu_Q(C)\mid Q\in C(\kbar)\}$.  Let $P\in C(\kbar)$ be a point with $\mu_P(C)=\mu$.  Let $\mu'=\max\{\mu_Q(C)\mid Q\in C(\kbar)\setminus \{P\}\}$ (we may have $\mu'=\mu$).  Then from \eqref{bleq}, \eqref{eqint}, and \eqref{iexc}, if $\phi(\tilde{Q})=Q\neq P$, then
\begin{equation*}
\ord_{\tilde{Q}}(\phi^*D)\leq d+(m-1)\mu'.
\end{equation*}
For points $\tilde{Q}$ with $\phi(\tilde{Q})=P$ we have
\begin{align*}
\sum_{\substack{\tilde{Q}\in C\\\phi(\tilde{Q})=P}}\max\{\ord_{\tilde{Q}}(\phi^*D)-d,0\}&=\sum_{\substack{Q'\in C'\\\pi(Q')=P}}\sum_{\substack{\tilde{Q}\in C\\\psi(\tilde{Q})=Q'}}\max\{\ord_{\tilde{Q}}(\phi^*D)-d,0\}\\
&\leq \sum_{\substack{Q'\in C'\\\pi(Q')=P}}(m-1)(C'.E)_{Q'}\leq (m-1)(C.E)\\
&\leq (m-1)\mu.
\end{align*}

It follows that we may write
\begin{equation*}
\phi^*D\leq (d+(m-1)\mu')F+G,
\end{equation*}
where $F$ is a reduced effective divisor on $\tilde{C}$ and $G$ is an effective divisor on $\tilde{C}$ satisfying
\begin{equation*}
\deg G\leq (m-1) (\mu-\mu').
\end{equation*}
Since $C$ is not a line, $\mu_Q(C)+\mu_{Q'}(C)\leq d$ for any distinct $Q,Q'\in C$ (look at the intersection number with the line through $Q$ and $Q'$).  So $\mu+\mu'\leq d$ and $\deg G\leq (m-1)(d-2\mu')$.  Let $A$ be a point on $\tilde{C}$.  By Theorem \ref{corV},
\begin{equation*}
m_{F,S}(P)<(2\delta+\epsilon)h_A(P)+O(1)
\end{equation*}
for all $P\in \tilde{C}(\kbar)\setminus \Supp F$ satisfying $[k(P):k]\leq \delta$.  We also have the bound
\begin{equation*}
m_{G,S}(P)<(\deg G+\epsilon)h_A(P)+O(1)
\end{equation*}
for all $P\in \tilde{C}(\kbar)\setminus \Supp G$.  Furthermore, $h(\phi(P))\geq (d-\epsilon)h_A(P)+O(1)$.
So
\begin{align*}
m_{\phi^*D,S}(P)&\leq (d+(m-1)\mu')m_{F,S}(P)+m_{G,S}(P)+O(1)\\
&\leq \left((d+(m-1)\mu')(2\delta)+(m-1)(d-2\mu')+\epsilon\right)h_A(P)+O(1)\\
&\leq \left(2d\delta+(m-1)((2\delta-2)\mu'+d)+\epsilon\right)h_A(P)+O(1)
\end{align*}
for all $P\in \tilde{C}(\kbar)\setminus \Supp \phi^*D$ satisfying $[k(P):k]\leq \delta$.
Since $\mu'\leq \frac{d}{2}$,
\begin{align*}
m_{D,S}(\phi(P))&=m_{\phi^*D,S}(P)+O(1)\leq (2d\delta+(m-1)d\delta+\epsilon)h_A(P)+O(1)\\
&\leq ((m+1)\delta d+\epsilon)h_A(P)+O(1)\leq ((m+1)\delta+\epsilon')h(\phi(P))+O(1),
\end{align*}
for all $P\in \tilde{C}(\kbar)\setminus \Supp \phi^*D$ satisfying $[k(P):k]\leq \delta$, where $\epsilon'\to 0$ as $\epsilon\to 0$.
\end{proof}

We now prove the main result of this section.

\begin{theorem}
\label{thmain}
Let $S$ be a finite set of places of a number field $k$.  Let $L_1,\ldots, L_q$ be lines over $k$ in $\mathbb{P}^2$ in general position.  Let $D=\sum_{i=1}^qL_i$ and let $\epsilon>0$.
\begin{enumerate}
\item  There exists a finite union of lines $Z\subset \mathbb{P}^2$ such that the inequality
\begin{equation}
\label{eqq1}
m_{D,S}(P)=\sum_{v\in S}\sum_{\substack{w\in M_{k(P)}\\w\mid v}} \sum_{i=1}^q\lambda_{L_i,w}(P)<\left(\frac{15}{2}+\epsilon\right)h(P)+O(1)
\end{equation}
holds for all $P\in \mathbb{P}^2(\kbar)\setminus Z$ satisfying $[k(P):k]\leq 2$.
\item  The inequality
\begin{equation}
\label{eqq2}
m_{D,S}(P)=\sum_{v\in S}\sum_{\substack{w\in M_{k(P)}\\w\mid v}} \sum_{i=1}^q\lambda_{L_i,w}(P)<(8+\epsilon)h(P)+O(1)
\end{equation}
holds for all $P\in \mathbb{P}^2(\kbar)\setminus \Supp D$ satisfying $[k(P):k]\leq 2$.
\end{enumerate}
\end{theorem}

It follows easily from Wirsing's theorem that the set $Z$ in part (a) may be taken to consist of the lines $L_1,\ldots, L_q$ together with the finite set of lines in $\mathbb{P}^2$ that pass through four distinct intersection points of the lines $L_1,\ldots, L_q$.

\begin{proof}
We first note that part (b) follows from part (a) and Theorem \ref{thl}(a).

We now prove part (a).  Let $\Sym^2\mathbb{P}^2$ denote the symmetric square of $\mathbb{P}^2$.  We have a natural map $\psi:\mathbb{P}^2\times\mathbb{P}^2\to \Sym^2\mathbb{P}^2$.  Let $\pi_i:\mathbb{P}^2\times\mathbb{P}^2\to\mathbb{P}^2$, $i=1,2$, denote the projection on to the $i$th factor.  For a line $L\subset \mathbb{P}^2$, let $L^{(2)}$ denote the divisor $\psi_*\pi_1^*L$ on $\Sym^2\mathbb{P}^2$.  Let $D^{(2)}=\sum_{i=1}^qL^{(2)}_i$, a divisor on $\Sym^2\mathbb{P}^2$.  By the same proof as for the inequality $C'(\delta, l,m,n)\leq B'(l,\delta m, \delta n)$ in Theorem \ref{compth}, to prove part (a) it suffices to prove the inequality
\begin{equation*}
m_{D^{(2)},S}(P)<\left(\frac{15}{2}+\epsilon\right)h_{L^{(2)}}(P)+O(1)
\end{equation*}
for all $P\in \Sym^2\mathbb{P}^2(k)\setminus (Y\cup \Supp D^{(2)})$ for some finite union of curves $Y$, where $L$ is any line in $\mathbb{P}^2$.  Actually, this only gives part (a) with $Z$ a finite union of curves, but the additional fact that $Z$ can be taken to consist only of lines follows from Theorem \ref{thl}(b) as $15/2\geq 6$.  Since at most four divisors $L^{(2)}_i$ meet at a point, any point $P\in \Sym^2\mathbb{P}^2(k)$ can be $v$-adically close to at most four divisors $L^{(2)}_i$.  It follows that for any point $P\in \Sym^2\mathbb{P}^2(k)\setminus \Supp D^{(2)}$, for some choice of divisors $L_{i,v}\in \{L_1,\ldots,L_q\}$ (depending on $P$) we have
\begin{equation*}
m_{D^{(2)},S}(P)\leq \sum_{v\in S}\sum_{i=1}^4 \lambda_{L^{(2)}_{i,v},v}(P)+O(1).
\end{equation*}
Thus, it suffices to show that for any choice of $L_{i,v}$ ($L_{1,v},\ldots,L_{4,v}$ distinct for fixed $v\in S$),
\begin{equation}
\label{mineq}
\sum_{v\in S}\sum_{i=1}^4 \lambda_{L_{i,v}^{(2)},v}(P)<\left(\frac{15}{2}+\epsilon\right)h_{L^{(2)}}(P)+O(1)
\end{equation}
for all $P\in \Sym^2\mathbb{P}^2(k)\setminus (Y\cup \Supp D^{(2)})$ for some finite union of curves $Y$.


We first show that \eqref{mineq} holds outside a set $W$, of a certain form, with $\dim W\leq 2$.  To accomplish this we will use the Subspace Theorem applied appropriately to a certain embedding of $\Sym^2\mathbb{P}^2$ in $\mathbb{P}^5$.  Specifically, consider $\mathbb{P}^2\times\mathbb{P}^2$ with coordinates $(x,y,z)\times (\tilde{x},\tilde{y},\tilde{z})$ and the morphism $\mathbb{P}^2\times\mathbb{P}^2\to \mathbb{P}^5$ given by $(x,y,z)\times (\tilde{x},\tilde{y},\tilde{z})\mapsto (x\tilde{x},y\tilde{y},z\tilde{z},x\tilde{y}+\tilde{x}y,x\tilde{z}+\tilde{x}z,y\tilde{z}+\tilde{y}z)$.  This induces an embedding $\phi:\Sym^2\mathbb{P}^2\hookrightarrow\mathbb{P}^5$.  In fact, this is just the embedding associated to the linear system $|L^{(2)}|$.

We will need the following elementary lemma.
\begin{lemma}
\label{lemcon}
Let $l_1,\ldots, l_6$ be linear forms in $k[x,y,z]$.  Let $P_i\in \mathbb{P}^2$, $i=1,\ldots,6$, denote the point dual to $l_i$ ($ax+by+cz\mapsto (a,b,c)\in \mathbb{P}^2$).  The polynomials $l_i(x,y,z)l_i(\tilde{x},\tilde{y},\tilde{z})$, $i=1,\ldots,6$, are linearly dependent over $k$ if and only if the points $P_1,\ldots,P_6$ lie on some (possibly reducible) conic in $\mathbb{P}^2$.  In particular, for any five distinct linear forms $l_1,\ldots,l_5$, such that no four of the corresponding lines $L_1,\ldots,L_5\subset \mathbb{P}^2$ meet at a point, the polynomials $l_i(x,y,z)l_i(\tilde{x},\tilde{y},\tilde{z})$, $i=1,\ldots,5$, are linearly independent.
\end{lemma}
\begin{proof}
Let $l_i(x,y,z)=a_ix+b_iy+c_iz$.  Expanding each polynomial $l_i(x,y,z)l_i(\tilde{x},\tilde{y},\tilde{z})$ and writing down the coefficients of $x\tilde{x},y\tilde{y},z\tilde{z},x\tilde{y}+\tilde{x}y,x\tilde{z}+\tilde{x}z,y\tilde{z}+\tilde{y}z$, we obtain an associated matrix $M$ with $ith$ row given by $(a_i^2,b_i^2,c_i^2,a_ib_i,a_ic_i,b_ic_i)$.  The polynomials $l_i(x,y,z)l_i(\tilde{x},\tilde{y},\tilde{z})$, $i=1,\ldots,6$, are linearly dependent over $k$ if and only if $M$ has a nontrivial nullspace.  Now we note that the vector $(d_1,\ldots,d_6)^{T}$ is in the right null-space if and only if the conic $d_1x^2+d_2y^2+d_3z^2+d_4xy+d_5xz+d_6yz$ contains the six points $P_1,\ldots,P_6$.

For the last statement of the lemma, note that the assumption implies that no four of the dual points $P_1,\ldots, P_5$ are collinear.  Since in this case there is a unique conic through $P_1,\ldots, P_5$, there is a sixth point $P_6$ such that $P_1,\ldots,P_6$ do not lie on any conic.  It follows that $l_i(x,y,z)l_i(\tilde{x},\tilde{y},\tilde{z})$, $i=1,\ldots,5$, must be linearly independent.
\end{proof}

We now show that \eqref{mineq} holds outside a set $W=\phi^{-1}(W')$, where $W'\subset \mathbb{P}^5$ is a finite union of linear spaces of dimension three.  This will be a consequence of the following lemma.

\begin{lemma}
\label{5lem}
Let $H_1,\ldots, H_q$ be hyperplanes in $\mathbb{P}^5$, defined over $k$, such that the intersection of any five of the hyperplanes is a point.  For each $v\in S$, let $H_{1,v},\ldots, H_{4,v}\subset \{H_1,\ldots, H_q\}$ be four distinct hyperplanes.  Let $\epsilon>0$.  Then
\begin{equation*}
\sum_{v\in S}\sum_{i=1}^4 \lambda_{H_{i,v},v}(P)<\left(\frac{15}{2}+\epsilon\right)h(P)+O(1)
\end{equation*}
for all $P\in \mathbb{P}^5(k)\setminus W'$, where $W'\subset \mathbb{P}^5$ is a finite union of linear spaces of dimension three.
\end{lemma}

\begin{proof}
From Schmidt's theorem, it is immediate that
\begin{equation*}
\sum_{v\in S}\sum_{i=1}^4 \lambda_{H_{i,v},v}(P)<\left(6+\epsilon\right)h(P)+O(1)
\end{equation*}
for all $P\in \mathbb{P}^5(k)$ outside a finite union of hyperplanes over $k$.  Let $X\cong \mathbb{P}^4\subset \mathbb{P}^5$ be a hyperplane in $\mathbb{P}^5$ not equal to any $H_i$, $i=1,\ldots, q$.  We will identify $X$ with $\mathbb{P}^4$.  Then it suffices to show that the inequality
\begin{equation*}
\sum_{v\in S}\sum_{i=1}^4 \lambda_{H_{i,v}|_X,v}(P)<\left(\frac{15}{2}+\epsilon\right)h(P)+O(1)
\end{equation*}
holds for all $P\in \mathbb{P}^4(k)$ outside a finite union of hyperplanes in $\mathbb{P}^4\cong X$.  This inequality follows trivially if $q\leq 7$, so we may suppose that $q>7$.  Let $H_i'=H_i|_X$.  We consider two cases:\\\\
Case I:  All of the hyperplanes $H_i'$ are distinct.\\

Let $r,s,t,u$ be distinct elements of $\{1,\ldots, q\}$.  Then we first claim that there is at most one three-element subset $I\subset \{r,s,t,u\}$ such that $\dim_{i\in I} H_i'\neq 1$.  On the contrary, suppose that, without loss of generality, $\dim H_r'\cap H_s'\cap H_t'>1$ and $\dim H_r'\cap H_s'\cap H_u'>1$.  Since $H_r'\neq H_s'$ by assumption, $\dim H_r'\cap H_s'=2$.  Thus, $H_t'\supset (H_r'\cap H_s')$ and $H_u'\supset (H_r'\cap H_s')$.  But then this implies that $\dim H_r'\cap H_s'\cap H_t'\cap H_u'=2$, which contradicts the assumption that the intersection of any five hyperplanes in $\{H_1,\ldots, H_q\}$ is a point.

For $v\in S$, if the hyperplanes $H_{1,v}', H_{2,v}', H_{3,v}', H_{4,v}'$ are not in general position, then we assume that the hyperplanes are ordered so that $\dim \cap_{i=1}^3 H_{i,v}'=2$.  Note that for $j=1,2,3$, the hyperplanes $H_{i,v}'$, $i\in \{1,2,3,4\}\setminus \{j\}$, are in general position.  Then using the Subspace Theorem for $\mathbb{P}^4$ three times, we find

\begin{equation*}
\sum_{j=1}^3\sum_{v\in S}\sum_{i\in \{1,2,3,4\}\setminus \{j\}} \lambda_{H_{i,v}',v}(P)<\left(15+\epsilon\right)h(P)+O(1)
\end{equation*}
for all $P\in \mathbb{P}^4(k)$ outside a finite union of hyperplanes.  Since for $P\in \mathbb{P}^4(k)\setminus \cup_{i=1}^qH_i'$,
\begin{equation*}
\sum_{j=1}^3\sum_{v\in S}\sum_{i\in \{1,2,3,4\}\setminus \{j\}} \lambda_{H_{i,v}',v}(P)> 2\sum_{v\in S}\sum_{i=1}^4 \lambda_{H_{i,v}',v}(P),
\end{equation*}
we obtain

\begin{equation*}
\sum_{v\in S}\sum_{i=1}^4 \lambda_{H_{i,v}',v}(P)<\left(\frac{15}{2}+\epsilon\right)h(P)+O(1)
\end{equation*}
for all $P\in \mathbb{P}^4(k)$ outside a finite union of hyperplanes, as desired.\\

\noindent
Case II:  $H_m'=H_n'$ for some $m\neq n$\\

Since any five distinct hyperplanes in $\{H_1,\ldots, H_q\}$ intersect in a point, it follows that for any subset $\{i,j,k\}\subset \{1,\ldots,q\}$ with $\{m,n\}\not\subset\{i,j,k\}$, $H_i',H_j',H_k'$ are in general position.  For $j=1,2,3,4$, define
\begin{equation*}
\mathcal{H}_{j,v}=
\begin{cases}
\{H_{1,v},H_{2,v},H_{3,v},H_{4,v}\}\setminus \{H_{j,v}\} &\text{ if } \{H_m,H_n\}\not\subset\{H_{1,v},H_{2,v},H_{3,v},H_{4,v}\},\\
\{H_{1,v},H_{2,v},H_{3,v},H_{4,v}\}\setminus \{H_n\} &\text{ if } \{H_m,H_n\}\subset\{H_{1,v},H_{2,v},H_{3,v},H_{4,v}\}.
\end{cases}
\end{equation*}

Then using the Schmidt Subspace Theorem four times and summing, we find that
\begin{equation*}
\sum_{j=1}^4\sum_{v\in S}\sum_{H\in \mathcal{H}_{j,v}} \lambda_{H',v}(P)<\left(20+\epsilon\right)h(P)+O(1)
\end{equation*}
for all $P\in \mathbb{P}^4(k)$ outside a finite union of hyperplanes.  From the definitions, for all $P\in \mathbb{P}^4(k)\setminus \cup_{i=1}^qH_i'$,
\begin{align*}
\sum_{j=1}^4\sum_{v\in S}\sum_{H\in \mathcal{H}_{j,v}} \lambda_{H',v}(P)&> 3\sum_{v\in S}\sum_{i=1}^4 \lambda_{H_{i,v}',v}(P)-2\sum_{v\in S}\lambda_{H_n',v}(P)+O(1)\\
&> 3\sum_{v\in S}\sum_{i=1}^4 \lambda_{H_{i,v}',v}(P)-2h(P)+O(1).
\end{align*}
Combining these equations appropriately, we find that
\begin{equation*}
\sum_{v\in S}\sum_{i=1}^4 \lambda_{H_{i,v}',v}(P)<\left(\frac{22}{3}+\epsilon\right)h(P)+O(1)<\left(\frac{15}{2}+\epsilon\right)h(P)+O(1)
\end{equation*}
for all $P\in \mathbb{P}^4(k)$ outside a finite union of hyperplanes.
\end{proof}

A polynomial $f\in \mathbb{Q}[x,y,z,\tilde{x},\tilde{y},\tilde{z}]$ that is homogeneous in the variables $x,y,z$ and the variables $\tilde{x},\tilde{y},\tilde{z}$ and invariant under interchanging $x,y,z$ with $\tilde{x},\tilde{y},\tilde{z}$, naturally defines a closed subset of $\Sym^2\mathbb{P}^2$.  More generally, if a system of polynomial equations defines a closed subset $Z$ of $\mathbb{P}^2\times \mathbb{P}^2$, then on $\Sym^2\mathbb{P}^2$ we will say that the system of polynomial equations defines the closed subset $\psi(Z)$, where $\psi:\mathbb{P}^2\times \mathbb{P}^2\to\Sym^2\mathbb{P}^2$ is the canonical map.  If $L\subset\mathbb{P}^2$ is defined by a linear form $l$, then $L^{(2)}$ is defined by $l(x,y,z)l(\tilde{x},\tilde{y},\tilde{z})$, and $L^{(2)}=\phi^*H$ for an appropriate hyperplane $H\subset \mathbb{P}^5$.  Let $H_1,\ldots, H_q$ be hyperplanes in $\mathbb{P}^5$ such that $L^{(2)}_i=\phi^*H_i$, $i=1,\ldots, q$.  By Lemma \ref{lemcon}, since $L_1,\ldots, L_q$ are in general position, for any five distinct lines $L_{i_1},\ldots, L_{i_5}\subset \{L_1,\ldots,L_q\}$, the corresponding polynomials $l_{i_j}(x,y,z)l_{i_j}(\tilde{x},\tilde{y},\tilde{z})$, $j=1,\ldots, 5$, are linearly independent, or equivalently, any five distinct hyperplanes in $\{H_1,\ldots, H_q\}$ intersect in a single point.  Then by Lemma~\ref{5lem},

\begin{align*}
\sum_{v\in S}\sum_{i=1}^4 \lambda_{L_{i,v}^{(2)},v}(P)&=\sum_{v\in S}\sum_{i=1}^4 \lambda_{H_{i,v},v}(\phi(P))+O(1)\\
&<\left(\frac{15}{2}+\epsilon\right)h(\phi(P))+O(1)\\
&<\left(\frac{15}{2}+\epsilon\right)h_{L^{(2)}}(P)+O(1),
\end{align*}
for all $P\in \Sym^2\mathbb{P}^2(k)\setminus W$, where $W=\phi^{-1}(W')$ and $W'\subset \mathbb{P}^5$ is a finite union of linear spaces of dimension three.

Since it's easily seen that for $W=\phi^{-1}(W')$ as above, $\dim W=2$, we have now reduced the problem to showing that \eqref{mineq} holds outside a finite union of curves for rational points on certain surfaces in $\Sym^2 \mathbb{P}^2$.  More specifically, let $V\subset \Sym^2 \mathbb{P}^2$ be a surface that is an irreducible component of $\phi^{-1}(W')$, where $W'$ is some linear subspace of $\mathbb{P}^5$ of codimension two.  We may assume that $V$ is defined over $k$, since otherwise $V(k)$ will lie in a finite union of curves in $V$.  We also assume that $V$ is not contained in the support of any divisor $L_i^{(2)}$.  Let $D_i=L_i^{(2)}|_V$, $D_{i,v}=L_{i,v}^{(2)}|_V$, and $A=L^{(2)}|_V$.  Then we need to show that
\begin{equation*}
\sum_{v\in S}\sum_{i=1}^4 \lambda_{D_{i,v},v}(P)<\left(\frac{15}{2}+\epsilon\right)h_{A}(P)+O(1)
\end{equation*}
for all $P\in V(k)\setminus Y$ for some finite union of curves $Y$.  We begin by showing that this inequality must hold if not too many of the divisors $D_{i,v}$ share components.
\begin{lemma}
\label{lc}
Suppose that for all $v\in S$, no three of the divisors $D_{i,v}$, $i=1,\ldots, 4$, have a common component.  Let $\epsilon>0$.  Then there exists a finite union of curves $Y$ in $V$ such that
\begin{equation*}
\sum_{v\in S}\sum_{i=1}^4 \lambda_{D_{i,v},v}(P)<\left(\frac{36}{5}+\epsilon\right)h_{A}(P)+O(1)
\end{equation*}
for all $P\in V(k)\setminus Y$.
\end{lemma}
\begin{proof}
As in the proof of Theorem \ref{degen}, by using the normalization of $V$ and functoriality, we easily reduce to the case that $V$ is a normal surface.

We have, for $v\in S$,
\begin{equation*}
\sum_{\substack{i,j\\i\neq j}}\lcm(D_{i,v},D_{j,v})=\sum_{\substack{i,j\\i\neq j}}D_{i,v}+D_{j,v}-\gcd(D_{i,v}, D_{j,v})=6\sum_{i=1}^4D_{i,v}-\sum_{\substack{i,j\\i\neq j}}\gcd(D_{i,v}, D_{j,v}).
\end{equation*}
Since no three of the divisors $D_{i,v}$, $i=1,\ldots, 4$, have a common component, it follows that for fixed $i$,
\begin{equation*}
\sum_{\substack{j\\j\neq i}}\gcd(D_{i,v}, D_{j,v})\leq D_{i,v}.
\end{equation*}
So
\begin{equation*}
\sum_{\substack{i,j\\i\neq j}}\gcd(D_{i,v},D_{j,v})\leq \sum_{i=1}^4D_{i,v}.
\end{equation*}
Then
\begin{equation*}
\sum_{\substack{i,j\\i\neq j}}\lcm(D_{i,v},D_{j,v})\geq 5\sum_{i=1}^4 D_{i,v}.
\end{equation*}.

Translating this information into an inequality of Weil functions and using \eqref{fundineq} multiple times, we obtain that for some finite union of curves $Y$,
\begin{equation*}
5\sum_{v\in S}\sum_{i=1}^4 \lambda_{D_{i,v},v}(P)\leq \sum_{v\in S} \sum_{\substack{i,j\\i\neq j}}\lambda_{\lcm(D_{i,v},D_{j,v}),v}(P)<(36+\epsilon)h_A(P)
\end{equation*}
for all $P\in V(k)\setminus Y$.
\end{proof}

By Lemma \ref{lc}, we may assume from now on that there exist three divisors $D_i$ that share a component.  After reindexing and choosing coordinates appropriately, we may assume that the three divisors are $D_1, D_2$, and $D_3$, that $L_1, L_2,$ and $L_3$ are defined by $x=0$, $y=0$, $z=0$, respectively (so $D_1,D_2,D_3$ are defined on $V\subset\Sym^2\mathbb{P}^2$ by $x\tilde{x}=0, y\tilde{y}=0, z\tilde{z}=0$, respectively), and that a common irreducible component of the three divisors is the curve $C\subset\Sym^2\mathbb{P}^2$ defined by the equation $x=y=\tilde{z}=0$.  So $V\subset\phi^{-1}(W')$ is a surface in $\Sym^2\mathbb{P}^2$ containing the curve $C$.  It follows that $V$ is an irreducible component of a closed set $W\subset \Sym^2\mathbb{P}^2$ defined by $f=g=0$ with $f$ and $g$ two polynomials of the form
\begin{align*}
f&=f_1x\tilde{x}+f_2y\tilde{y}+f_3z\tilde{z}+f_4(x\tilde{y}+y\tilde{x}),\\
g&=g_1x\tilde{x}+g_2y\tilde{y}+g_4(x\tilde{y}+y\tilde{x}),
\end{align*}
where $f_i,g_i\in k$, $i=1,2,3,4$.
We consider various cases depending on the form of $f$ and $g$.\\

Case I:  $f_3=0$\\\\

We have
\begin{align*}
f&=(f_1\tilde{x}+f_4\tilde{y})x+(f_4\tilde{x}+f_2\tilde{y})y,\\
g&=(g_1\tilde{x}+g_4\tilde{y})x+(g_4\tilde{x}+g_2\tilde{y})y.
\end{align*}
Let
\begin{equation*}
(f_1x+f_4y)(g_4x+g_2y)-(f_4x+f_2y)(g_1x+g_4y)=l_1(x,y)l_2(x,y).
\end{equation*}
The equation $f=g=0$ implies that either $x=y=0$ or $l_1(\tilde{x},\tilde{y})l_2(\tilde{x},\tilde{y})=0$.  Similarly, it also implies that either $\tilde{x}=\tilde{y}=0$ or $l_1(x,y)l_2(x,y)=0$.  It follows easily that $W$ has two irreducible components.  One of the irreducible components is defined by $x=y=0$ in $\Sym^2\mathbb{P}^2$. The other irreducible component is defined by $l_1(x,y)=l_2(\tilde{x},\tilde{y})=0$ in $\Sym^2\mathbb{P}^2$.  The curve $C:x=y=\tilde{z}=0$ only belongs to the component defined by $x=y=0$.  However, this irreducible component is contained in the support of $L_1^{(2)}$, and hence is excluded from our analysis.\\

Case II: $g=l(x,y)l(\tilde{x},\tilde{y})$ for some linear form $l$ and $f_3\neq 0$\\

Let $W$ be the variety in $\mathbb{P}^2\times \mathbb{P}^2$ defined by $f=l(x,y)=0$.  The map $\psi$ induces a surjective morphism (that we will again denote by $\psi$) $\psi:W\to V$ that is injective outside the closed subset defined by $f=l(x,y)=l(\tilde{x},\tilde{y})=0$.  We will consider $\psi^*D_i=F_i+\tilde{F}_i$, where $F_i=(L_i\times \mathbb{P}^2)\cap W$ and $\tilde{F}_i=(\mathbb{P}^2\times L_i)\cap W$.  Let $L\subset \mathbb{P}^2$ be the line defined by $l(x,y)=0$.  We have two projection maps $\pi_1:W\to L\cong \mathbb{P}^1$ and $\pi_2:W\to \mathbb{P}^2$ induced by the two natural projections on $\mathbb{P}^2\times \mathbb{P}^2$.  


Note that $f_1x\tilde{x}+f_2y\tilde{y}+f_4(x\tilde{y}+y\tilde{x})$ is not a multiple of $g$ since this would imply that $V$ is contained in the closed subset of $\Sym^2\mathbb{P}^2$ defined by $z\tilde{z}=0$, i.e., $V\subset \Supp L_3^{(2)}$.  Let $P=(\tilde{x}_0,\tilde{y}_0,0)$ be the unique point in $\mathbb{P}^2$ such that when $f$ is evaluated at $(\tilde{x},\tilde{y},\tilde{z})=(\tilde{x}_0,\tilde{y}_0,0)$, the result is a scalar multiple of $l(x,y)$.  Then $\pi_2$ is an isomorphism above $\mathbb{P}^2\setminus \{P\}$ and $\pi_2^{-1}(P)=L\times \{P\}\cong \mathbb{P}^1$.  In fact, the map $\pi_2$ realizes $W$ as the blow-up of $\mathbb{P}^2$ at $P$.  The map $\pi_1$ gives $W$ as a rational ruled surface.


Since the lines $L_i$ are in general position, it follows that aside from the line $L_3$ given by $z=0$, there is at most one other line $L_j$ passing through $P$.  We will consider the case where $L_j\neq L_1, L_2$ (the cases $L_j=L_1$ and $L_j=L_2$ are very similar).  In this case, without loss of generality (adding an appropriate line if necessary), we can assume that the line $L_4$ contains the point $P$.

Let $r,s,t,u\in \{1,\ldots,q\}$ be distinct.  Since $L_1,\ldots, L_q$ are in general position, there exists a partition $\{r,s,t,u\}=\{i_1,i_2\}\cup \{i_3,i_4\}$ such that $L_{i_1}\cap L_{i_2}\cap L=\emptyset, L_{i_3}\cap L_{i_4}\cap L=\emptyset$, and $\{i_1,i_2\}, \{i_3,i_4\}\neq \{3,4\}$.  

Suppose first that $\{i_1,i_2\}\cap \{3,4\}=\emptyset$.  Then we claim that $\psi^*D_{i_1}$ and $\psi^*D_{i_2}$ have no components in common.  First, from the definitions it follows that for all $i$, $F_i=\pi_1^*(L\cap L_i)$ and $\tilde{F}_i=\pi_2^*L_i$.  Since $L_{i_1}\cap L_{i_2}\cap L=\emptyset$, $F_{i_1}$ and $F_{i_2}$ have no component in common.  Note that $F_i\cong \mathbb{P}^1$ is irreducible for all $i$ and if $i\neq 3,4$, then $\tilde{F}_i$ is irreducible since $P\not\in L_i$.  More precisely, since $\tilde{F}_i=\pi_2^*L_i$, if $i\neq j$ then $\tilde{F}_i$ and $\tilde{F}_j$ can share a component only in the case $\{i,j\}=\{3,4\}$, in which case they share the exceptional divisor $E$ (with respect to the map $\pi_2$).  Thus $\tilde{F}_{i_1}$ and $\tilde{F}_{i_2}$ do not have a component in common.  Finally, note that for any $i$, $\pi_2(F_i)$ is a line containing $P$.  So $F_i$ and $\tilde{F}_j$ do not have a component in common if $\{i,j\}\cap \{3,4\}=\emptyset$.  It follows that $\psi^*D_{i_1}$ and $\psi^*D_{i_2}$ do not have any components in common.

Suppose now that, say, $\{i_1,i_2\}=\{3,i\}$.  The same argument as above shows that $F_i$ and $F_3$ have no common component, $\tilde{F}_i$ and $\tilde{F}_3$ have no common component, and $\tilde{F}_i$ and $F_3$ have no common component.  However, it may happen that $F_i$ and $\tilde{F}_3$ have a component in common.  We have $\tilde{F}_3=L_3'+E$, where $L_3'$ is the strict transform of $L_3$ (with respect to the map $\pi_2$) and $L_3'$ may be a common component of $F_i$ and $\tilde{F}_3$.  A similar conclusion holds if $4\in \{i_1,i_2\}$.

It follows from the above that $\gcd(\psi^*D_{i_1},\psi^*D_{i_2})+\gcd(\psi^*D_{i_3},\psi^*D_{i_4})\leq L_3'+L_4'$.  Using this fact and applying \eqref{fundineq} twice, we find that
\begin{equation*}
\sum_{v\in S}\sum_{i=1}^4 \lambda_{\psi^*D_{i,v},v}(P)-m_{L_3',S}(P)-m_{L_4',S}(P)<\left(6+\epsilon\right)h_{\psi^*A}(P)+O(1)
\end{equation*}
for all $P\in W(k)$ outside a finite union of curves.  Note that $\psi^*A-L_3'-L_4'\sim E$, which is effective.  So 
\begin{equation*}
m_{L_3',S}(P)+m_{L_4',S}(P)<h_{L_3'}(P)+h_{L_4'}(P)+O(1)<h_{\psi^*A}(P)+O(1)
\end{equation*}
for all $P\in W(\kbar)\setminus (E\cup L_3'\cup L_4')$.  Then
\begin{equation*}
\sum_{v\in S}\sum_{i=1}^4 \lambda_{\psi^*D_{i,v},v}(P)<\left(7+\epsilon\right)h_{\psi^*A}(P)+O(1)
\end{equation*}
for all $P\in W(k)$ outside a finite union of curves.  Since $\psi$ is birational, using functoriality we obtain that
\begin{equation*}
\sum_{v\in S}\sum_{i=1}^4 \lambda_{D_{i,v},v}(P)<\left(7+\epsilon\right)h_A(P)+O(1),
\end{equation*}
for all $P\in V(k)$ outside a finite union of curves as was to be shown.\\

Case III:  The general case (not Case I or Case II).\\

Consider $W=\psi^{-1}(V)\subset \mathbb{P}^2\times \mathbb{P}^2$ and the projection on to the first factor $\pi_1:W\to \mathbb{P}^2$.  
The map $\pi_1$ is one-to-one above $(x,y,z)$ unless the matrix
\begin{equation*}
\left(
\begin{array}{ccc}
f_1x+f_4y & f_4x+f_2y & f_3z\\
g_1x+g_4y & g_4x+g_2y & 0
\end{array}
\right)
\end{equation*}
has rank one.  Since we are not in Case I or Case II, $g_1x+g_4y$ and $g_4x+g_2y$ are linearly independent and $f_3\neq 0$.  This implies that $\pi_1$ is one-to-one above $(x,y,z)$ unless $x=y=0$ or $z=(f_1x+f_4y)(g_4x+g_2y)-(f_4x+f_2y)(g_1x+g_4y)=0$.  Similarly, the projection onto the second factor $\pi_2:W\to \mathbb{P}^2$ is one-to-one above $(\tilde{x},\tilde{y},\tilde{z})$ unless $\tilde{x}=\tilde{y}=0$ or $\tilde{z}=(f_1\tilde{x}+f_4\tilde{y})(g_4\tilde{x}+g_2\tilde{y})-(f_4\tilde{x}+f_2\tilde{y})(g_1\tilde{x}+g_4\tilde{y})=0$.

We assume now that $z=0, (f_1x+f_4y)(g_4x+g_2y)-(f_4x+f_2y)(g_1x+g_4y)=0$, defines two distinct points (the case where it defines a single point is even easier).  After possibly extending $k$, we can assume without loss of generality that the two points are $k$-rational.  Let $P=(0,0,1)\in \mathbb{P}^2$ and let $Q,R\in\mathbb{P}^2$ be the two distinct points defined by $z=0, (f_1x+f_4y)(g_4x+g_2y)-(f_4x+f_2y)(g_1x+g_4y)=0$.  Since the lines $L_i$ are in general position, aside from the line $L_3$ given by $z=0$, there is at most one line $L_j$ passing through $Q$ and at most one line $L_k$ passing through $R$.  We will prove the case $\{L_j,L_k\}\cap \{L_1,L_2\}=\emptyset$ (the case $\{L_j,L_k\}\cap \{L_1,L_2\}\neq\emptyset$ being almost entirely similar).  Then we may assume, without loss of generality, that the line $L_4$ contains $Q$ and the line $L_5$ contains $R$ (we may add such lines to $\{L_1,\ldots, L_q\}$ if they don't exist).


We claim that:
\begin{enumerate}
\item  For $i\geq 6$, no divisor $D_i$ has a component in common with any other divisor $D_j$, $j\neq i$.\label{3a}
\item  $D_3$ has components in common with each $D_i$, $i\leq 5$.\label{3b}
\item  $D_1$, $D_2$, and $D_3$ all have a component in common.\label{3c}
\item  If $1\leq i<j\leq 5$ and $D_i$ and $D_j$ are not as in \eqref{3b} or \eqref{3c}, then $D_i$ and $D_j$ have no components in common.\label{3d}
\end{enumerate}

We will consider the projections of $\psi^*D_i=E_i+\tilde{E}_i$, where $E_i=(L_i\times \mathbb{P}^2)\cap W$ and $\tilde{E}_i=(\mathbb{P}^2\times L_i)\cap W$.  We first note that $\pi_1(E_i)=L_i$ and $\pi_2(\tilde{E}_i)=L_i$.  Recall that $\pi_1$ and $\pi_2$ are isomorphisms above $\mathbb{P}^2\setminus \{P,Q,R\}$ (in fact, $\pi_1$ and $\pi_2$ both realize $W$ as the blow-up of $\mathbb{P}^2$ at the three points $P,Q,$ and $R$;  the map $\pi_2\circ \pi_1^{-1}$ is a plane Cremona transformation).  In particular, it follows that for $i\geq 6$, $E_i$ and $\tilde{E}_i$ are irreducible.  For two points $P_1,P_2\in \mathbb{P}^2$, let $L_{P_1P_2}$ denote the line through $P_1$ and $P_2$.  Then we claim that $\pi_1^{-1}(P)=\{P\}\times L_{QR}=\{P\}\times L_3, \pi_1^{-1}(Q)=\{Q\}\times L_{PR},$ and $\pi_1^{-1}(R)=\{R\}\times L_{PQ}$.  From the definitions of $P, Q,$ and $R$, it's clear that $\pi_1^{-1}(P), \pi_1^{-1}(Q),$ and $\pi_1^{-1}(R)$ are all of the form $\{{\rm pt}\} \times L$ for some point ${\rm pt}$ and some line $L$.  Then a direct calculation using the relevant equations shows that $(P,Q), (P,R),(Q,R)\in W$, and our claim follows.  Note that 
\begin{equation*}
\pi_1^{-1}(P)\cap \tilde{E}_i=(\{P\}\times L_{QR})\cap (\mathbb{P}^2\times L_i)=\{P\}\times (L_{QR}\cap L_i) \neq \emptyset.  
\end{equation*}
Similarly, $\pi_1^{-1}(T)\cap \tilde{E}_i\not=\emptyset$ for $T\in \{P,Q,R\}$ and so $P,Q,R\in \pi_1(\tilde{E}_i)$ for all $i$ (in fact, for $i\geq 6$, $\pi_1(\tilde{E}_i)$ is an irreducible conic containing $P,Q,$ and $R$).  Therefore, $\pi_1(\tilde{E}_i)\not\subset \pi_1(E_j)$ for $i\neq j$.  Since we also have  $\pi_1(E_i)\not\subset \pi_1(E_j)$, $i\neq j$, this implies that if $i\geq 6$ and $i\neq j$, then $\psi^*D_i$ and $\psi^*D_j$ have no irreducible components in common.  Part \eqref{3a} follows.  The divisors $E_1,E_2,\tilde{E}_3$ share the component $x=y=\tilde{z}=0$.  The divisors $E_3$ and $E_4$ share the component $\pi_1^{-1}(Q)$ and $E_3$ and $E_5$ share the component $\pi_1^{-1}(R)$.  This proves \eqref{3b} and \eqref{3c}.  By the same argument as for \eqref{3a}, one sees that $D_4$ doesn't share components with $D_1$, $D_2$, or $D_5$, and similarly $D_5$ doesn't share components with $D_1$ or $D_2$.

Now, if we exclude the divisor $D_3$, then only $D_1$ and $D_2$ share a component.  It follows that for any distinct $r,s,t,u\in \{1,\ldots,q\}$, there is a partition $\{r,s,t,u\}=\{i_1,i_2\}\cup \{i_3,i_4\}$, such that $\lcm(D_{i_1},D_{i_2})+\lcm(D_{i_3},D_{i_4})\geq D_i+D_j+D_k+D_l-D_3$.  Appropriately applying \eqref{fundineq} twice, we then obtain
\begin{equation*}
\sum_{v\in S}\sum_{i=1}^4 \lambda_{D_{i,v},v}(P)-m_{D_3,S}(P)<\left(6+\epsilon\right)h_A(P)+O(1)
\end{equation*}
for all $P\in V(k)$ outside a finite union of curves.  Since
\begin{equation*}
m_{D_3,S}(P)<h_{D_3}(P)+O(1)=h_A(P)+O(1),
\end{equation*}
we find that
\begin{equation*}
\sum_{v\in S}\sum_{i=1}^4 \lambda_{D_{i,v},v}(P)<\left(7+\epsilon\right)h_A(P)+O(1).
\end{equation*}
for all $P\in V(k)$ outside a finite union of curves.

\end{proof}

\section{Conjectures}
\label{scon}

In this section we discuss some conjectural Schmidt-Wirsing type inequalities and some related examples.  For simplicity, we state conjectures only for hyperplanes in general position in projective space.  This special case seems representative and is likely to be quite difficult already.

We now state the main conjecture, which as noted below is not essentially new.

\begin{conjecture}
\label{conj}
Let $S$ be a finite set of places of a number field $k$.  Let $H_1,\ldots, H_q$ be hyperplanes over $k$ in $\mathbb{P}^n$ in general position.  Let $D=\sum_{i=1}^qH_i$.  Let $\delta$ be a positive integer and let $\epsilon>0$.
\begin{enumerate}
\label{gena}
\item  There exists a proper Zariski-closed subset $Z\subset \mathbb{P}^n$ such that the inequality
\begin{equation*}
m_{D,S}(P)<\left(2\delta+n-1+\epsilon\right)h(P)+O(1)
\end{equation*}
holds for all $P\in \mathbb{P}^n(\kbar)\setminus Z$ satisfying $[k(P):k]\leq \delta$.
\item  The inequality
\label{genb}
\begin{equation*}
m_{D,S}(P)<(2\delta n+\epsilon)h(P)+O(1)
\end{equation*}
holds for all $P\in \mathbb{P}^n(\kbar)\setminus \Supp D$ satisfying $[k(P):k]\leq \delta$.
\end{enumerate}
\end{conjecture}

Part (a) of the conjecture follows easily from Vojta's conjecture, which in this context states that there exists a proper Zariski-closed subset $Z\subset \mathbb{P}^n$ such that
\begin{equation}
\label{Vineq}
m_{D,S}(P)<(n+1+\epsilon)h(P)+d(P)+O(1)
\end{equation}
for all $P\in \mathbb{P}^n(\kbar)\setminus Z$ satisfying $[k(P):k]\leq \delta$.  Here $d(P)=\frac{1}{[\mathbb{Q}(P):\mathbb{Q}]}\log |D_{\mathbb{Q}(P)/\mathbb{Q}}|$, where $D_{k/\mathbb{Q}}$ denotes the absolute discriminant of $k$.  Part (a) now follows from \eqref{Vineq} and an inequality of Silverman \cite{Sil2} that states that
\begin{equation*}
d(P)\leq (2\delta-2)h(P)+O(1)
\end{equation*}
for all points $P\in \mathbb{P}^n(\kbar)$ satisfying $[k(P):k]\leq \delta$.  Part (b) of the conjecture was conjectured by Ru (Conjecture 4.4 of \cite{Ru2} with $l=1$).

We now give some examples related to the conjecture.  Let $\delta$ be a positive integer.  We will repeatedly use the fact that for any $2\delta$ points $P_1,\ldots, P_{2\delta}\in \mathbb{P}^1(k)$ and any finite set of places $S$ of $k$ containing the archimedean places with $|S|\geq 2$, there exists an infinite set of points $R\subset \{P\in \mathbb{P}^1(\kbar)\mid [k(P):k]\leq \delta\}$ that is $S$-integral with respect to $P_1,\ldots, P_{2\delta}$, i.e., for each $i$,
\begin{equation*}
m_{P_i,S}(P)=h(P)+O(1), \quad \forall P\in R,
\end{equation*}
or equivalently,
\begin{equation*}
\sum_{i=1}^{2\delta}m_{P_i,S}(P)=2\delta h(P)+O(1), \quad \forall P\in R.
\end{equation*}
Explicitly, if $\phi:\mathbb{P}^1\to\mathbb{P}^1$ is a morphism over $k$ satisfying $\phi^*(0)=\sum_{i=1}^\delta P_i$ and $\phi^*(\infty)=\sum_{i=\delta+1}^{2\delta} P_i$, then we may take $R=\phi^{-1}(\co_{k,S}^*)$, identifying $\co_{k,S}^*\subset \mathbb{A}^1\subset \mathbb{P}^1$ in the usual way.

We first give an example showing that, in contrast to the Schmidt Subspace Theorem, one must allow varieties other than hyperplanes in the exceptional set $Z$ in part (a).

\begin{example}
\label{excon}
Let $L_1,\ldots, L_8\subset \mathbb{P}^2$ be lines over a number field $k$ in general position and $C$ a conic over $k$ in $\mathbb{P}^2$ such that $C$ contains $L_i\cap L_{i+4}$, $i=1,2,3,4$, and $L_i$ is tangent to $C$ for $i=1,2,3,4$.  Let $\{P_i\}=L_i\cap C=L_{i+4}\cap C$, $i=1,2,3,4$.  Let $S$ be a finite set of places of $k$ containing the archimedean places and satisfying $|S|\geq 2$.  Then the conic $C$ contains an infinite set of quadratic points $R$ that are $S$-integral with respect to $P_1,\ldots, P_4$, that is,
\begin{equation*}
\sum_{i=1}^4m_{P_i,S}(P)=4h_Q(P)+O(1)
\end{equation*}
and $[k(P):k]\leq 2$ for all $P\in R$, where $Q$ is some $k$-rational point on $C$.  Let $\iota:C\to \mathbb{P}^2$ be the inclusion morphism.  Then for all $P\in R$,
\begin{equation*}
\sum_{i=1}^8m_{L_i,S}(\iota(P))=\sum_{i=1}^8m_{\iota^*L_i,S}(P)=3\sum_{i=1}^4m_{P_i,S}(P)=12h_Q(P)+O(1).
\end{equation*}
Since $h(\iota(P))=2h_Q(P)+O(1)$ for all $P\in C(\kbar)$, we find that 
\begin{equation*}
\sum_{i=1}^8m_{L_i,S}(\iota(P))=6h(\iota(P))+O(1)
\end{equation*}
for all $P\in R$.  It follows that $C$ must be in the exceptional set $Z$ in Conjecture \ref{conj} (a) for the lines $L_1,\ldots, L_8$ and $\delta=2$.
\end{example}

We now give examples showing that the quantities $2\delta+n-1$ and $2\delta n$ in the conjectures cannot be replaced by any smaller numbers.

\begin{example}
Let $\delta$ and $n\geq 2$ be positive integers.  Let $H_1,\ldots, H_{n+2\delta-1}\subset \mathbb{P}^n$ be hyperplanes in general position over a number field $k$.  Let $\{P\}=\cap_{i=1}^nH_i$ and let $L$ be a line through $P$ over $k$.  Let $\cup_{i=1}^{n+2\delta-1}(L\cap H_i)=\{P_1,\ldots, P_t\}$, where $t\leq 2\delta$.  Let $S$ be a finite set of places of $k$ containing the archimedean places and satisfying $|S|\geq 2$.  Let $R\subset \mathbb{P}^1(\kbar)$ be an infinite set of points satisfying $m_{P_i,S}(P)=h(P)+O(1)$, $i=1,\ldots, t$, and $[k(P):k]\leq \delta$ for all $P\in R$.  Then
\begin{equation*}
\sum_{i=1}^{n+2\delta-1}m_{H_i,S}(P)=(n+2\delta-1)h(P)+O(1)
\end{equation*}
for all $P\in R$.  So there exist infinitely many lines $L\subset \mathbb{P}^n$, with Zariski-dense union in $\mathbb{P}^n$, each containing infinitely many points $P$ satisfying
\begin{equation*}
\sum_{i=1}^{n+2\delta-1}m_{H_i,S}(P)=(n+2\delta-1)h(P)+O_L(1)
\end{equation*}
and $[k(P):k]\leq \delta$.  It follows that for any $\epsilon>0$, there exists a Zariski-dense set of points $P\in \mathbb{P}^n(\kbar)$ satisfying
\begin{equation*}
\sum_{i=1}^{n+2\delta-1}m_{H_i,S}(P)>(n+2\delta-1-\epsilon)h(P)
\end{equation*}
and $[k(P):k]\leq \delta$.
\end{example}

\begin{example}
\label{exinf}
Let $\delta$ and $n\geq 2$ be positive integers.  Let $H_1,\ldots, H_{2\delta n}\subset \mathbb{P}^n$ be hyperplanes in general position over a number field $k$ such that $\bigcap_{i=(j-1)n+1}^{jn}H_i=\{P_j\}$ consists of a single point for $j=1,\ldots,2\delta$, and $P_1,\ldots, P_{2\delta}$ all lie on a line $L$ over $k$.  Then by the same argument as in the previous example there exists a finite set of places $S$ of $k$ and an infinite set of points $R\subset L(\kbar)\subset \mathbb{P}^n(\kbar)$ such that 
\begin{equation*}
\sum_{i=1}^{2\delta n}m_{H_i,S}(P)=2\delta nh(P)+O(1) 
\end{equation*}
and $[k(P):k]\leq \delta$ for all $P$ in $R$.  
\end{example}

In view of Example \ref{excon}, one may also ask for a version of Conjecture \ref{conj} where $Z$ can be taken to be a finite union of hyperplanes.  For $\mathbb{P}^2$, Theorem \ref{thl} suggests:
\begin{conjecture}
Let $S$ be a finite set of places of a number field $k$.  Let $L_1,\ldots, L_q$ be lines over $k$ in $\mathbb{P}^2$ in general position.  Let $D=\sum_{i=1}^qL_i$.  Let $\delta$ be a positive integer and let $\epsilon>0$.  There exists a finite union of lines $Z\subset \mathbb{P}^2$ such that the inequality
\begin{equation*}
m_{D,S}(P)<\left(3\delta+\epsilon\right)h(P)+O(1)
\end{equation*}
holds for all $P\in \mathbb{P}^2(\kbar)\setminus Z$ satisfying $[k(P):k]\leq \delta$.
\end{conjecture}

Example \ref{excon} shows that, at least for $\delta=2$, the quantity $3\delta$ in the conjecture cannot be replaced by a smaller number.

\bibliography{AlgebraicSchmidt}
\end{document}